\documentclass[12pt]{article}
\usepackage{geometry}
\usepackage{fancyhdr}
\usepackage{titlesec}
\usepackage{lipsum}
\usepackage{graphicx}
\usepackage{setspace}
\usepackage{amsmath}
\usepackage{amssymb}
\usepackage{bm}
\usepackage{array}
\usepackage{times}
\usepackage{tabularx}
\usepackage{titletoc}

\usepackage{mathtools}
\usepackage{multirow,array}
\usepackage{graphicx}
\usepackage{subfigure}
\usepackage{epstopdf}
\usepackage{float} 
\usepackage{color, xcolor}
\usepackage{amsthm}

\usepackage{comment}
\def\theequation{\arabic{section}.\arabic{equation}}
\usepackage{amssymb,amsmath,amsfonts,latexsym} 
\usepackage{hyperref}
   
\usepackage{mathtools}
\usepackage{multirow,array}
\usepackage{graphicx}
\usepackage{subfigure}
\usepackage{epstopdf}
\usepackage{float} 
\usepackage{color, xcolor}
\usepackage{comment}
\allowdisplaybreaks

\newtheorem{theorem}{Theorem}[section]
\newtheorem{lemma}[theorem]{Lemma}
\newtheorem{corollary}[theorem]{Corollary}
\newtheorem{definition}[theorem]{Definition}

\newcommand{\caputo}[3][0]{{}_{#1}^{C}\! D^{#2}_{#3}\!}
\newcommand{\RL}[3][0]{{}_{#1}^{RL}\! D^{#2}_{#3}\!}
\newcommand{\J}[3][0]{{}_{#1}\! J^{#2}_{#3}\!}

\newcommand{\Conv}{%
  \mathop{\scalebox{1.5}{\raisebox{-0.2ex}{$\circledast$}}
  }
}

\newcommand{\R}{\mathbb{R}}
\newcommand{\C}{\mathbb{C}}
\newcommand{\Z}{\mathbb{Z}}

\newcommand{\El}{\mathcal{L}}

\begin{document}


\title{Convolution-to-sum identities for Mittag-Leffler type functions}

\author{
    William Cvetko\thanks{University of Utah. \texttt{wilcvetko@gmail.com}}
    \and
    Elena Cherkaev\thanks{Department of Mathematics, University of Utah, 
    155 South 1400 East, JWB 233, Salt Lake City, UT. \texttt{elena@math.utah.edu}}
}



\maketitle

\begin{abstract}

 Product-to-sum identities for trigonometric functions play a fundamental role in function theory and numerous applications. In this spirit, we present
convolution-to-sum identities for Mittag-Leffler type functions. 
Using a Laplace domain analysis of fractional operators, we identify a family of Mittag-Leffler type functions that encapsulates the eigenfunctions of Riemann-Liouville and Caputo fractional derivatives. We work with two closely-related parameterizations of this class, $R_{\alpha,v}$ and $P_{\alpha,w}$.
The convolution of two such functions can be expressed as a series of them. Moreover, if the functions share the same order $\alpha$, the convolution can be reduced to a sum of two $P/R$ functions through a partial-fraction decomposition in the Laplace domain. Furthermore, $R$ and $P$ functions satisfy a generalization of Euler's identity, which expands the scope of the previous result to convolutions of $P/R$ functions whose orders $\alpha_1,\alpha_2$ are related by a rational factor. For $\frac{\alpha_1}{\alpha_2} = \frac{n}{m}$, the resulting sum has $n+m$ terms. The foundational results and methods developed here are illustrated by their application to forced subdiffusion and to a fractionally attenuated wave equation (the Caputo-Wismer-Kelvin, or the fractional Kelvin-Voigt model).

\end{abstract} 

\smallskip
\noindent\textbf{Keywords:} fractional calculus (primary); 
Mittag-Leffler type functions; 
fractional ordinary and partial differential equations; 
convolution identities; 
Laplace transform methods; 
fractional eigenfunctions

\smallskip
\noindent\textbf{MSC 2020:} 26A33 (primary); 33E12; 34A08; 35R11; 44A10


\section{Introduction} \label{sec:1}

\setcounter{section}{1} \setcounter{equation}{0} 
Fractional differential equations incorporating the Caputo ($\caputo[]{\alpha}{}$) and Riemann-Liouville ($\RL[]{\alpha}{}$) derivatives capture complex phenomena owing to their nonlocal memory effects and interpolative nature between standard derivatives. These properties have found use across areas such as materials science, viscoelasticity, heat transfer,  and epidemiology \cite{Bagley_Torvik_1986_fractional_viscoelastic,Caputo_Mainardi_1971_Dissipation_model,Kontou_2025_polymer_viscoplastic,Lvov_2021_time_fractional_phase_field,Mozafarifard2020_time_frac_subdiffusion_thin_metal,Prasad_Kumar_Dohare_2023_Zika,Riechers2024_Cluster_dynamics_Fractional_diffusion_mettalic_glass,Voller_2018_anomalous_heat_transfer}. In context of these applications, there has also been much work on numerical methods, including existence and regularity proofs, inverse problems,  and further generalization of differential operators, for a spread of models both linear and nonlinear \cite{Borikhanov_Ruzhansky_Torebek_2023_qualitative,Ervin_2021_regularity_fractional_diffusion_advection_etc,Gal_Warma_2020_fractional_in_time_parabolic_eq_books,Li_Khan_Riaz_2024_fractional_stochastic,Kaltenbacher_Rundell_2023_inverse_problems,Li_Celik_Telyakovskiy_2024_completely_nonlocal,Li_Huang_Liu_2023_coupled_systems}.

 An important class of problems is related to the properties of the eigenfunctions of Caputo and Riemann-Liouville fractional derivatives, that share a similar form described by Mittag-Leffler functions, $E_{\alpha,\beta}(z)$ (see Definition \ref{def:mittag_leffler_2_param}); this function has been called "the queen of fractional calculus" \cite{Mainardi_2020_Mittag_Leffler_Queen_Fractional_Calc}. For an eigenvalue $\lambda\in \C$, the eigenfunctions of the Caputo and Riemann-Liouville fractional-differential operators of order $\alpha >0$ are respectively given by (\cite{Grigoletto_2018_fractional_eigenfunctions} as well as the present work)
\[ 
t^{j} E_{\alpha,j+1}(\lambda t^\alpha)  \qquad \mbox{and} \qquad 
 t^{\alpha-j-1 }E_{\alpha,\alpha-j}(\lambda t^\alpha),
\]
 for integers $j$ such that $0 \leq j < \lceil\alpha\rceil$, where $\lceil\alpha\rceil$ is the ceiling of $\alpha$. The properties of functions of this form, which largely stem from their elegant Laplace transform and their power-law decay for $\lambda<0,\alpha<1$, have been studied in \cite{Aleroev_2025_fractional_sturm_Liouville,Apelblat_2020_differentiation_ML,Gorenflo_Kilbas_Mainardi_Rogosin_2014_ML_book,Kiryakova_2000_multiindex_ML,Mainardi_2000_mittag_type_functions_in_frac_evolution,Mainardi1996_Fractional_Relaxation_oscillation,Mainardi_2020_Mittag_Leffler_Queen_Fractional_Calc,Podlubny1999}. Notable work includes the numerical evaluation of such functions in \cite{Garrappa_2015_numerical_evaluation_ML} and the asymptotic limits of them \cite{Mainardi2014_MittagLeffler_Properties}. 
 Much of the literature has focused on the case where $\alpha <1, $ or $1<\alpha<2$, typically extending models which employ first or second time derivatives. However the methods and results in this paper are applicable for all $\alpha>0$.
 In the current work, we study parametrizations of Mittag-Leffler type functions,  $P_{\alpha,w}(\lambda,t)$ and $R_{\alpha,v}(\lambda,t)$, given by 
 \[
  P_{\alpha,w}(\lambda,t) = t^{w} E_{\alpha,w+1}(\lambda t^\alpha)  \qquad \mbox{and} \qquad R_{\alpha,v}(\lambda,t) = t^{\alpha-v-1}E_{\alpha,\alpha-v}(\lambda t^\alpha),
 \]
 which, respectively, are eigenfunctions of the Caputo and Riemann-Liouville derivatives of order $\alpha$ when $w$ and $v$ are integers with $w,v\in \{0,\cdots,\lceil\alpha\rceil-1\}$. This class of functions will hereafter be referred to as '$P/R$ functions' for brevity.
Note that when $\alpha = 1$, they are exponentials, and when $\alpha = 2$, they are trigonometric or hyperbolic cosine and sine functions. For example, $P_{1,0}(a,t) =e^{at}$ and $ R_{2,1}(-\omega^2,t) = \cos(\omega t)$. The $R$ function was introduced in \cite{LorenzoHartley1999_generalized_functions_for_fractional}, subsequently explored in \cite{LorenzoHartley2000_RFunction_relationships,lorenzo2016_fractional_Trig_book},  and is referred to as the Lorenzo-Hartley function. As $R_{\alpha,j}(\lambda,t)$ gives the $j$th eigenfunction for the Riemann-Liouville function, we parallel this convention by introducing $P$ such that $P_{\alpha,j}(\lambda,t)$ is the $j$th eigenfunction for the Caputo derivative. For further justification, note that for $j=0$, $\lambda= -1$, $P_{\alpha,0}(-1,t) = E_\alpha(-t^\alpha)$, which is among the most studied forms of the Mittag-Leffler function, often denoted as $e_\alpha(t)$ (particularly in Mainardi's work, for instance  \cite{Mainardi2014_MittagLeffler_Properties}).


In the current paper, we prove convolution-to-sum identities for Mittag-Leffler type functions. These results generalize the representations available for the convolution of exponential functions. For example, the convolution of exponential functions $e^{at}*e^{bt}$ is given as $\frac{1}{(s-a)(s-b)}$ in the Laplace domain, which is represented as $\frac{1}{a-b}(\frac{1}{s-a}-\frac{1}{s-b})$ using partial fraction decomposition, leading to the sum of these functions $\frac{1}{a-b}(e^{at}-e^{bt})$ in the time domain. Using the Laplace transform of Mittag-Leffler type functions, we develop similar representations for convolutions of $P/R$ functions.
This property is foundational for Mittag-Leffler functions. Especially elegantly, it generalizes to convolutions of $P/R$ functions of orders $\alpha_1,\alpha_2$, whose quotient is rational. In this case, we show that their convolution can be evaluated as a finite linear combination of $P/R$ functions. This serves as an efficient computational tool for solving fractional-differential models with nonzero forcing terms and for other applications. Furthermore, even when $\frac{\alpha_1}{\alpha_2}$ is irrational, the convolution of two $P/R$ functions can be written as a series of $P/R$ functions. Additionally, we present an important generalization of Euler's identity to Mittag-Leffler functions, which is fundamental to understanding how Mittag-Leffler functions of different orders relate.


This paper is organized as follows. In Section \ref{sec:2} we define the Riemann-Liouville and Caputo fractional derivatives and derive their eigenfunctions, motivating the class of $P/R$ functions. In Section \ref{sec:3} we derive a general convolution-to-series identity for $P/R$ functions, a generalization of Euler's formula, along with convolution-to-sum identities for $P/R$ functions with rationally-related orders. In Section \ref{sec:4}, we demonstrate the strength of the results in Section \ref{sec:3} by applying them to an inhomogeneous anomalous diffusion model and an attenuated wave model (Caputo-Wismer-Kelvin/fractional Kelvin-Voigt model). The work concludes in Section \ref{sec:5}, with Appendix \ref{secA1} containing definitions and results related to the Laplace transform, which are referenced throughout the paper.

 \section{Review of fractional operators and eigenfunctions} \label{sec:2}

\setcounter{section}{2} \setcounter{equation}{0} 

\subsection{\bf Operator definitions}
\label{subsec:2.definitions}

The conventions for fractional calculus in this paper are built on the Riemann-Liouville fractional integral. A history of the conventions of fractional calculus can be found in \cite{Ross1975_History_of_Fractional_calc}.
 \begin{definition}
\label{def:fractionalIntegral}
Provided it converges, the order-$\nu \, (\nu>0)$ Riemann-Liouville fractional integral of $f$ based at $t_0\in\R$ is defined as
\begin{equation}
    \J[t_0]{\nu}{} [f](t) :=  \frac{1}{\Gamma(\nu)} \int _{t_0}^t \!\!f(z) (t-z)^{\nu-1} dz \label{eq:fractionalIntegral}
\end{equation}
For $\nu=0$, $J^\nu$ is taken as the identity operator.
\end{definition}
In this paper we assume $t_0 = 0$, and suppress the subscript.
In addition to other values of $t_0$, there are further conventions for the bounds on the fractional integral, such as $(-\infty,t), (-\infty,+\infty),(t,\infty)$ among others, which can be found in \cite{de_Oliveira_2014_review_of_Defns}.

 \begin{definition}
 \label{def:derivativeRL}
Provided the fractional integral converges, the order-$\alpha$ Riemann-Liouville and Caputo fractional derivatives of $f$ based at $0$ are defined as
     \begin{equation}
    \RL[]{\alpha}{}[f](t) := D^{\lceil \alpha \rceil} \J[]{\lceil\alpha\rceil-\alpha}{} [f](t)
    \label{eq:derivativeRL}
\end{equation}
\begin{equation}
    \caputo[]{\alpha}{}[f](t) :=  \J[ ]{\lceil\alpha\rceil-\alpha}{} [D^{\lceil \alpha \rceil}f](t)\label{eq:derivativeCaputo}
\end{equation}
respectively, where $\lceil\alpha\rceil$ is the ceiling of $\alpha$ ($\alpha$ rounded up to the next integer). For $\alpha<0$, both derivatives are taken as $\J[]{ -\alpha}{}[f](t).$
\end{definition}

 The approach to fractional operators exploited in this paper is based on the Laplace transform \cite{Gorenflo_Mainardi_1997_Integral_differential_eqs_fractional_order,Oldham_Spanier_1974_fractional_calculus}. The properties of the Laplace transform used here are presented in the Appendix.

 \begin{lemma}\label{lemma:laplace_int}
     For $\nu \geq 0$, if $f$ has a Laplace transform, then \begin{equation}
     \El[\J[]{\nu}{}[f](t)](s) = s^{-\nu} \El[f](s)
 \end{equation}
 \end{lemma}
 \begin{proof}
 Note that $\J[]{\nu}{}[f](t) = \frac{1}{\Gamma(\nu)}\int_0^tf(\tau) (t-\tau)^{\nu-1}  d \tau $ is a causal convolution of $f$ with $\frac{t^{\nu-1}}{\Gamma(\nu)}$. By the convolution theorem (\ref{thm:ConvolutionTheorem} in the Appendix), \[
 \El\left[f*\frac{t^{\nu-1}}{\Gamma(\nu)}\right] = \El[f] \El \left[ \frac{t^{\nu-1}}{\Gamma(\nu)}\right],
 \]
 and the Laplace transform of $\frac{t^{\nu-1}}{\Gamma(\nu)}$ is  $s^{-\nu}$ (see \ref{eq:LaplaceOfPower} ). If $\nu=0$, $J^\nu$ is the identity and the result follows trivially.
 \end{proof}
 
Using Lemma \ref{lemma:laplace_int} and the properties of $\frac{d}{dt}$ under the Laplace transform (\ref{eq:LaplaceOfD}, \ref{eq:LaplaceOfDn}) we can derive the Laplace transform of the Caputo and Riemann-Liouville derivatives.
\begin{lemma}\label{thm:LaplaceOfCaputo}
    Assuming $f^{(\ell)}$ is continuous on $[0,\infty)$ for $\ell=0,1,\ldots,\lceil\alpha\rceil-1$, the order-$\lceil\alpha\rceil$ derivative of $f$ is continuous on $(0,\infty),$ and the Laplace transform of $f(t)$  is $F(s)$, then
    \begin{equation}
        \El[\caputo[]{\alpha}{} f](s) = s^\alpha F(s) - \sum_{\ell = 0}^{\lceil\alpha\rceil-1} s^{\alpha-1-\ell}f^{(\ell)}(0) \label{eq:LaplaceOfCaputo}
    \end{equation}
\end{lemma}
\begin{proof}
    
\begin{align}
    \El[\caputo[]{\alpha}{} f](s) =& \El \left[ \J[]{\lceil\alpha\rceil-\alpha}{} \,\frac{d^{\lceil\alpha\rceil}}{dt^{\lceil\alpha\rceil} }f \right] \notag
\\ =& s^{\alpha - \lceil\alpha\rceil} \left(s^{\lceil\alpha\rceil}F(s) - \sum_{\ell = 0}^{\lceil\alpha\rceil-1} s^{\lceil\alpha\rceil-1-\ell}f^{(\ell)}(0)  \right) \notag 
\\=& s^\alpha F(s) - \sum_{\ell = 0}^{\lceil\alpha\rceil-1} s^{\alpha-1-\ell}f^{(\ell)}(0)  
\end{align}
\end{proof}

\begin{lemma}\label{thm:LaplaceOfRL}
Let $g(t) = \J[]{\lceil\alpha\rceil-\alpha}{}[f](t)$. If $g$ and its derivatives up to order $\lceil\alpha\rceil-1$ are continuous on $[0,\infty)$, the order-$\lceil\alpha\rceil$ derivative of $g$ is continuous on $(0,\infty)$, and the Laplace transform of $f$ is $F(s) $ then
    \begin{equation}\label{eq:LaplaceOfRL}
        \El[ \RL[ ]{\alpha}{}f] = s^\alpha F(s) -\sum_{k=0}^{\lceil\alpha\rceil-1} s^k (\RL[]{\alpha-1-k}{}f) |_{t=0_+}  =
        s^\alpha F(s) - \sum_{\ell=0}^{\lceil\alpha\rceil-1} s^{\lceil\alpha\rceil-1-\ell}g^{(\ell)}(0)  
    \end{equation}
\end{lemma}
\begin{proof}
    By definition, the Laplace transform of the RL derivative of $f$ is
\begin{align*}
     \El \left[   \frac{d^{\lceil\alpha\rceil}}{dt^{\lceil\alpha\rceil} }  \J[]{\lceil\alpha\rceil-\alpha}{}f\right](s)
     =& \El \left[   \frac{d^{\lceil\alpha\rceil}}{dt^{\lceil\alpha\rceil} } g \right](s)              \notag
    \\ =& s^{\lceil \alpha\rceil  } \El[g](s)  - \sum_{\ell=0}^{\lceil\alpha\rceil-1} s^{\lceil\alpha\rceil-1-\ell}  g^{(\ell)}(0)
    \\=&s^{\alpha} F(s) - \sum_{\ell=0}^{\lceil\alpha\rceil-1} s^{\lceil\alpha\rceil-1-\ell}  \left(\frac{d^\ell}{dt^\ell} (\J[]{\lceil\alpha\rceil-\alpha}{} f )\right)|_{t=0_+} \notag
    \\=&s^\alpha F(s) - \sum_{\ell=0}^{\lceil\alpha\rceil-1} s^{\lceil\alpha\rceil-1-\ell}  \left( \RL[]{\ell +\alpha-\lceil\alpha\rceil}{}f \right)|_{t=0_+}
\end{align*}
and to get its final form, we reverse the sum with $k=\lceil\alpha\rceil-1-\ell.$ Lastly note that for $k= \lceil\alpha\rceil-1$, the order of the RL derivative is $\alpha-\lceil\alpha\rceil <0$. This is simply $\J[]{\lceil\alpha\rceil-\alpha}{}$.

\end{proof}
\textbf{Remark.} Note that the straightforward Laplace transform of the Caputo derivatives is analogous to that of integer derivatives, only depending on the initial value of $f$ and its integer derivatives. The same is not true for the Riemann-Liouville derivatives, which explicitly depend on the initial fractional derivatives of $f$. This can lead to singularities at $t=0$, among other problems. For physical interpretations of these fractional initial values, see \cite{Heymans_2005_Physical_interpretation_of_RL}. The comparative clarity of Caputo initial conditions is a major reason for their widespread use, though introduced later than the Riemann-Liouville derivative.

\subsection{\bf Eigenfunctions and function definitions}

For both Caputo and RL derivatives, consider the eigenfunction equation \begin{equation}
    D^\alpha f = \lambda f \label{eq:eigenfunction_equation}.
\end{equation}
The corresponding eigenfunctions are derived via the Laplace transform. We introduce the following functions, which characterize the eigenfunctions of the fractional derivatives.
\begin{definition} \label{def:mittag_leffler_2_param}
For $z\in \C$, $\alpha,\beta\in \R_{>0}$, The Mittag-Leffler function \cite{Wiman_1905_Uber_den_Fundame_etc} is defined as
    \begin{equation}
        E_{\alpha,\beta}(z) := \sum_{n=0}^\infty \frac{z^n}{\Gamma(\alpha n + \beta)}.
    \end{equation}
\end{definition}
The convergence of this series for the range of parameters stated is due to the Gamma function in the denominator.
Following Lorenzo and Hartley in \cite{LorenzoHartley1999_generalized_functions_for_fractional,lorenzo2016_fractional_Trig_book}, 
we define the Mittag-Leffler type functions $R_{\alpha,v}(\lambda,t)$:
\begin{definition} \label{def:R_func}
For $\lambda \in \C,$ $\alpha>0,$ $ v<\alpha,$ $ t>0$ ( $t \geq 0 $ for $v \leq \alpha-1$), the $R$ function is defined as
    \begin{equation}
        R_{\alpha,v}(\lambda,t) := t^{\alpha-v-1}E_{\alpha,\alpha-v}(\lambda t^\alpha) = \sum_{n=0}^\infty \frac{\lambda ^n t^{\alpha(n+1)-v-1}}{\Gamma(\alpha(1+n)-v)} .
    \end{equation}
\end{definition}
%
We show below that the Mittag-Leffler type functions $R_{\alpha,v}(\lambda,t)$ are eigenfunctions of the Riemann-Liouville fractional derivative operator. To parallel the $R$-function convention, we define functions $P_{\alpha,w}(\lambda,t)$: 
\begin{definition}\label{def:P_func} For $\lambda \in \C,$ $\alpha >0,$ $w>-1,$ $t >0$ ($t \geq 0$ for $w\geq 0$)
    \begin{equation}
        P_{\alpha,w}(\lambda,t) := t^w E_{\alpha,w+1}(\lambda t^\alpha) = \sum_{n=0}^\infty   \frac{\lambda^n t^{\alpha n + w}}{\Gamma(\alpha n+w+1)}
    \end{equation}
\end{definition} 
The $P$ functions are the eigenfunctions of the Caputo fractional derivative operator.\\
\textbf{ Remark.}
Note that the two forms of the Mittag-Leffler function, $R$ and $P$, are related to each other:
\begin{equation}\label{eq:P_and_R}
\begin{split}
    R_{\alpha,v}(\lambda,t) = P_{\alpha,\alpha-v-1}(\lambda,t)
    \\
    P_{\alpha,w}(\lambda,t) = R_{\alpha,\alpha-w-1}(\lambda,t)
\end{split}
\end{equation}

Finally, we derive the Laplace transform of the $R$ and $P$ functions:
\begin{lemma} \label{lemma:Laplace_of_R_P}
   The Laplace transform of functions $P$ and $R$, given in Definitions \ref{def:P_func} and \ref{def:R_func}, with parameters $\alpha>0,\lambda\in\C,v<\alpha,w>-1$ are given by 
    \begin{equation} \label{eq:Laplace_of_P_R}
        \El[R_{\alpha,v}(\lambda,t)](s) =\frac{s^v}{s^\alpha-\lambda}, \quad \El[P_{\alpha,w}(\lambda,t)](s) =\frac{s^{\alpha-1-w}}{s^\alpha-\lambda}.
    \end{equation}
\end{lemma}
\begin{proof}
    We will prove this for the case of $P$, and use \ref{eq:P_and_R} to derive the Laplace transform for $R$ functions. The series representation of $P_{\alpha,w}(\lambda,t)$ is 
    \[
    P_{\alpha,w}(\lambda,t) = \sum_{n=0}^\infty \frac{\lambda^n t^{n\alpha +w}}{\Gamma(n\alpha+w+1)}.
    \]
Using Lemma \ref{lemma:LaplaceOfPower} in the Appendix,
the Laplace transform of the $n$th term in the series is 
\begin{equation}
    \El\left[ \frac{ \lambda^n  t^{n\alpha+w}  }{\Gamma(n \alpha+w+1)}  \right](s) = \lambda^n s^{-1-n\alpha -w} = s^{-1-w}(\lambda s^{-\alpha})^n.
\end{equation}
    Hence $P_{\alpha,w}(\lambda,t)$ corresponds to the series
    \[
    s^{-1-w}\sum_{n=0}^\infty (\lambda s^{-\alpha})^n
    \]
    in the Laplace domain. To demonstrate that this geometric series converges, note that $\alpha>0,$ and so for $\Re[s] > |\lambda|^{1/\alpha}$ we have that $|\lambda s^{-\alpha}|<1$. Hence,
    \begin{equation}  \label{eq:Laplace_of_P}
        \El\left[P_{\alpha,w}(\lambda,t) \right ](s) = \frac{s^{-w-1}}{1-\lambda s^{-\alpha}} = \frac{s^{\alpha-w-1}}{s^\alpha-\lambda}
    \end{equation}
    as claimed.

    For the Laplace transform of $R_{\alpha,v}(\lambda,t),$ we use \ref{eq:Laplace_of_P} and  \ref{eq:P_and_R}. Indeed, from \ref{eq:P_and_R}, $R_{\alpha,v}(\lambda,t)= P_{\alpha,\alpha-v-1}(\lambda,t)$, it follows directly from \ref{eq:Laplace_of_P} that 
    \begin{equation}
        \El[R_{\alpha,v}(\lambda,t)] = \frac{s^v}{s^\alpha-\lambda}.
    \end{equation}

\end{proof}

With these results, we now prove that $P$ and $R$ functions characterize the eigenfunctions of the Caputo and Riemann-Liouville fractional derivatives.

\begin{theorem}
\label{thm:solution_to_eigenproblem}
    For $\alpha>0,\lambda\in\C,t\in (0,\infty),$ the solution (smooth for $t>0$) to the eigenfunction equation for eigenvalue $\lambda$, $D^\alpha f(t) = \lambda f(t)$, is given by
\begin{equation}
    \label{eq:eigenfunction_solutions}
    \begin{split}
        f_{RL}(t) =& \sum_{j=0}^{\lceil\alpha\rceil-1} c_j R_{\alpha,j}(\lambda,t)
        \\
        f_{C}(t) =& \sum_{j=0}^{\lceil\alpha\rceil-1} c_j P_{\alpha,j}(\lambda,t)
    \end{split}
\end{equation}
    for the Riemann-Liouville ($\RL[]{\alpha}{}$) and Caputo $(\caputo[]{\alpha}{})$ fractional derivative operators, respectively. Here, the coefficients $c_j \in \C$  specify the initial values of the derivatives of the eigenfunction.
\end{theorem}

\begin{proof}
    We begin by calculating the Laplace-space representation of the eigenfunction, $f_C$, of the Caputo fractional derivative operator. 
Because it satisfies $\caputo[]{\alpha}{}f_C = \lambda f_C,$ taking the Laplace transform of both sides gives
\begin{align*}
\lambda \El[f_C](s) =&   \El \left[ \caputo[]{\alpha}{} f_C \right].
\end{align*}
Using Lemma \ref{thm:LaplaceOfCaputo} (see \ref{eq:LaplaceOfCaputo}) to represent the Caputo derivative in the Laplace domain:
\begin{align*}
     \lambda \El[f_C](s) =&  
    s^\alpha \El[f_C](s)  - \sum_{j = 0}^{\lceil\alpha\rceil-1} s^{\alpha-1-j}f_C^{(j)}(0) .
\end{align*}
Solving for $\El[f_C]$, we then have
\begin{equation}
    \\
    \El[f_C](s)  =   \sum_{j = 0}^{\lceil\alpha\rceil-1} f^{(j)}(0) \frac{s^{\alpha-1-j}}{s^\alpha-\lambda}.
\end{equation}
Similarly, for $f_{RL}$, the eigenfunction of the Riemann-Liouville derivative, using Lemma \ref{thm:LaplaceOfRL} (see \ref{eq:LaplaceOfRL}):

\begin{align*}
    \lambda \El[f_{RL}](s) =&\El\left[\RL[]{\alpha}{}f_{RL}\right] (s)  
    \\
     \lambda \El[f_{RL}](s) =& s^\alpha \El[f_{RL}](s) - \sum_{j=0}^{\lceil\alpha\rceil-1} s^{j}  \left(  \RL[]{\alpha-1-j}{}f_{RL}\right)\big\vert_{t=0} .
\end{align*}
Solving for $f_{RL}$ gives that
\begin{equation}
   \El[f_{RL}](s) = \sum_{j=0}^{\lceil\alpha\rceil-1}  \left(   \RL[]{ \alpha -j-1}{}  f_{RL}\right)\big|_{t=0}   \frac{s^j}{s^\alpha-\lambda}.
\end{equation}
Ultimately, using coefficients $c_j$ to encapsulate  initial conditions,
\begin{equation}
    \El\left[f_C\right] = \sum_{j = 0}^{\lceil\alpha\rceil-1} c_j \frac{s^{\alpha-1-j}}{s^\alpha-\lambda}, \quad \El\left[f_{RL}\right] = \sum_{j = 0}^{\lceil\alpha\rceil-1} c_j \frac{s^{j}}{s^\alpha-\lambda}.
\end{equation}
Using Lemma \ref{lemma:Laplace_of_R_P} to represent the functions $f_C$ and $f_{RL}$ in the time domain gives
\begin{equation}
    f_C(t) = \sum_{j=0}^{\lceil\alpha\rceil-1} c_j P_{\alpha,j}(\lambda,t), \qquad f_{RL}(t) = \sum_{j=0}^{\lceil\alpha\rceil-1} c_j R_{\alpha,j}(\lambda,t)
\end{equation}
as claimed.
\end{proof}

\begin{corollary}
    For $\alpha>0,$ $\lambda \in \C$, $t>0,$ $k \in \{0,\cdots,\lceil\alpha\rceil-1\}$,  $P_{\alpha,k}(\lambda,t)$ and $R_{\alpha,k}(\lambda,t)$ are eigenfunctions of the Caputo and Riemann-Liouville fractional derivatives, respectively:
    \begin{equation}
        \caputo[]{\alpha}{} P_{\alpha,k}(\lambda,t) = \lambda P_{\alpha,k}(\lambda,t), \quad \RL[]{\alpha}{}R_{\alpha,k}(\lambda,t) = \lambda R_{\alpha,k}(\lambda,t)
    \end{equation}
\end{corollary}
\begin{proof}
    Take all the coefficients $c_j=0$ for $j\neq k$ in Theorem \ref{thm:solution_to_eigenproblem} and $c_k=1.$
\end{proof}

\textbf{Remark.} Many applications of fractional calculus assume the order of the derivative is between zero and one, $\alpha \in (0,1)$. In this case, both the Caputo and Riemann-Liouville derivatives have a one-dimensional eigenspace for each eigenvalue. Thus, $E_{\alpha}(\lambda t^\alpha) = P_{\alpha,0}(\lambda,t)$ and $t^{\alpha-1}E_{\alpha,\alpha}(\lambda t^\alpha) = R_{\alpha,0}(\lambda,t)$ are often referred to as 'the' eigenfunction of the Caputo or Riemann-Liouville derivative. Generally, however, the eigenspace of the order-$\alpha$ derivative is $\lceil\alpha\rceil-$dimensional.

\begin{figure}[t]
  \centering
\includegraphics[width=0.42\linewidth]{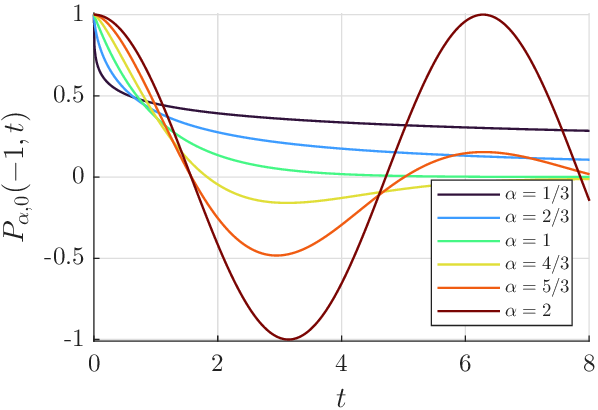}
\,
\includegraphics[width=0.42\linewidth]{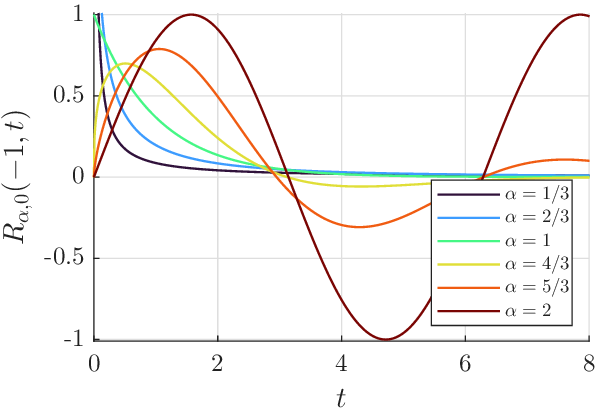}
    \caption{First eigenfunction of the $\alpha$-th Caputo (left) and Riemann–Liouville (right) fractional derivatives, with eigenvalue $-1$, across several values of $\alpha$. Note that the $P_{\alpha,0}(-1,t)$ all initially equal 1, and that for $0<\alpha<1$, $R_{\alpha,0}(-1,t)$ is singular at $t=0$. Mittag-Leffler functions in all plots have been calculated using the methods in \cite{Garrappa_2015_numerical_evaluation_ML}, implemented in MATLAB \cite{Garrappa_2025_MittagLeffler_MATLAB}.}
    \label{fig:eigenfunctionsComparison}
\end{figure}

\section{Convolution to sum identities} \label{sec:3}

\setcounter{section}{3} \setcounter{equation}{0} 

This section discusses convolutions of these Mittag-Leffler type $P$ and $R$ functions.  
We will show that the convolutions of these functions can be efficiently represented as their sums or series. 
The Laplace Convolution Theorem (Theorem \ref{thm:ConvolutionTheorem}) demonstrates that a causal convolution (Definition \ref{def:CausalConvolution}) in the time domain is a product in the Laplace domain, so, using the Laplace transform of the $P/R$ functions (\ref{eq:Laplace_of_P_R}), the Laplace transform of the convolution of the $R$ functions is represented as:
\begin{equation}
\El[R_{\alpha,v_1}(\lambda,t)*R_{\beta,v_2}(\sigma,t)] (s) = \frac{s^{v_1+v_2}}{(s^\alpha-\lambda)(s^\beta-\sigma)} .\label{eq:convolution_RqRp}
\end{equation}
We will exploit this form to simplify these convolutions and represent them as a series of $R$ functions.
\subsection{\bf Convolution to series expansion}

\begin{theorem}
    For $\alpha,\beta>0;\,v_1<\alpha,v_2<\beta;\,w_1,w_2>-1$ (to ensure convergence of the Laplace transform), $\lambda, \sigma \in \C$
    \begin{equation}
        \begin{split}
            R_{\alpha,v_1}(\lambda,t) *R_{\beta,v_2}(\sigma,t) =& \sum_{k=0}^\infty \sigma ^k R_{\alpha,v-\beta(k+1)}(\lambda,t) =\sum_{j=0}^\infty \lambda^j R_{\beta,v-\alpha(j+1)}(\sigma,t)
            \\
            P_{\alpha,w_1}(\lambda,t)*P_{\beta,w_2}(\sigma,t) =&\sum_{k=0}^\infty \sigma ^k P_{\alpha,w+\beta k+1}(\lambda,t) = \sum_{j=0}^\infty \lambda^j P_{\beta,w+\alpha j+1}(\sigma,t)
        \end{split}
    \end{equation}
    where $v=v_1+v_2,$ and $w=w_1+w_2$, a convention which will be adopted throughout the rest of this paper.
\end{theorem}
We use geometric expansion in $s$ and take term-by-term inverse Laplace transform to go back to the time domain, similarly to the steps used in the proof of Lemma \ref{lemma:Laplace_of_R_P}. 
\begin{proof}

By the Laplace convolution theorem (\ref{thm:ConvolutionTheorem} in the Appendix), the Laplace transform of $f*g$ results in a product $\El[f]\El[g]$. The Laplace representation of the convolution of the $R$ functions, $R_{\alpha,v_1}(\lambda,t)*R_{\beta,v_2}(\sigma,t)$, is $\frac{s^{v_1+v_2}}{(s^\alpha-\lambda)(s^\beta-\sigma)}$. As $\alpha,\beta>0$, after rewriting this as $s^{v-\alpha-\beta} (1-\lambda s^{-\alpha})^{-1}(1-\sigma s^{-\beta})^{-1}$, geometric expansion in both factors is valid  for $s > \max(|\lambda|^{1/\alpha},|\sigma|^{1/\beta})$:
    \begin{equation}
        s^{v-\alpha-\beta} (1-\lambda s^{-\alpha})^{-1}(1-\sigma s^{-\beta})^{-1}  = \sum_{j,k=0}^\infty \lambda^j \sigma^k s^{v-\alpha(j+1)-\beta(k+1)}.
    \end{equation}
Taking the inverse Laplace transform of each term leads to the following representation: 
\begin{align*}
    R_{\alpha,v_1}(\lambda,t)*R_{\beta,v_2}(\sigma,t) =& \sum_{j,k=0}^\infty \frac{\lambda^j \sigma^k t^{\alpha(j+1)+\beta(k+1)-v-1}}{\Gamma(\alpha(j+1)+\beta(k+1)-v)}
    \\
    =&\sum_{k=0}^\infty \sigma^k R_{\alpha,v-\beta(k+1)}(\lambda,t).
\end{align*}
By symmetry, this can also be expressed as $\sum_{j}^\infty \lambda^j R_{\beta,v-\alpha(j+1)}(\sigma,t)$. 

Similarly, because $\El[P_{\alpha,w}(\lambda,t)] = \frac{s^{\alpha-w-1}}{s^\alpha-\lambda}$, the Laplace transform of the convolution of $P$ functions $\El[P_{\alpha,w_1}(\lambda,t)*P_{\beta,w_2}(\sigma,t)]$ equals
\begin{align*}
   \El[P_{\alpha,w_1}(\lambda,t)*P_{\beta,w_2}(\sigma,t)]
    =&\frac{s^{\alpha+\beta-w-2}}{(s^\alpha-\lambda)(s^\beta-\sigma)}
   \\
    =& \sum_{j,k=0}^\infty \lambda^j \sigma^k s^{-\alpha j-\beta k-w-2}.
\end{align*}
Taking the inverse Laplace transform results in \begin{align*}
   P_{\alpha,w_1}(\lambda,t)*P_{\beta,w_2}(\sigma,t)= &\sum_{j,k=0}^\infty \lambda^j \sigma^k \frac{t^{\alpha j+\beta k+w +1}}{\Gamma(\alpha j+\beta k+w+2)}
    \\
    =&\sum_{k=0}^\infty \sigma^k P_{\alpha,w+\beta k+1}(\lambda,t)
\end{align*}

\end{proof}

Functions $R$ and $P$ are particular cases of the Mittag-Leffler function, so one might expect that the Mittag-Leffler function $E_{\alpha,\beta}(t)$ itself similarly has a simple identity for  convolutions of the form $E_{\alpha_1,\beta_1}(t)*E_{\alpha_2,\beta_2}(t)$. Generally, this is not the case because there is no simple Laplace transform when $E$ is not in the $P/R$ form. However, we can express the derived identity for the functions $P$ in terms of the Mittag-Leffler function:
\begin{equation}
    t^{w_1} E_{\alpha,w_1+1}(\lambda t^\alpha)* t^{w_2}E_{\beta,w_2+1}(\sigma t^\beta) = \sum_{k=0}^\infty \sigma^k t^{w+\beta k+1}E_{\alpha,w+\beta k+2}(\lambda t^\alpha ). \label{eq:E_expansion_inverse}
\end{equation}

These identities hold generally for any $\alpha,\beta>0$. With further constraints on $\alpha,\beta$, the infinite series of $P/R/E$ functions can be reduced to a finite sum of them. The strategy uses partial fraction decomposition in the Laplace domain, as suggested in chapter 2 of \cite{lorenzo2016_fractional_Trig_book}.

\subsection{\bf Convolution of the same-order functions}
 
A convolution of two $P/R$ functions whose orders are identical $\alpha=\beta$, can be simplified to a sum of two $P/R$ functions of order $\alpha$.

\begin{theorem} \label{thm:ConvolutionRqRq}
For $\lambda \neq \sigma$, $t>0$ and $\alpha>0,$  
$\alpha > v_k, ~ k=1,2,$ 
    \begin{equation}
        R_{\alpha,v_1}(\lambda, t) * R_{\alpha,v_2}(\sigma,t) = \frac{1}{\lambda-\sigma}\big(\lambda R_{\alpha,v-\alpha}(\lambda,t)-\sigma R_{\alpha,v-\alpha}(\sigma, t)\big) \label{eq:ConvolutionRqRq} 
    \end{equation}
   
    \begin{equation}
        P_{\alpha,w_1}(\lambda,t)*P_{\alpha,w_2}(\sigma,t) = \frac{1}{\lambda-\sigma}\bigg(\lambda P_{\alpha,w+1}(\lambda,t)-\sigma P_{\alpha,w+1}(\sigma,t) \bigg)
    \end{equation}

\begin{equation}\begin{split}
    &t^{w_1}E_{\alpha,w_1+1}(\lambda t^\alpha)*t^{w_2}E_{\alpha,w_2+1}(\sigma t^\alpha) \\&= \frac{1}{\lambda-\sigma}\bigg(  \lambda t^{w+1}E_{\alpha,w+2}(\lambda t^\alpha) - \sigma t^{w+1}E_{\alpha,w+2}(\sigma t^\alpha)\bigg).\end{split}
\end{equation}
Here $w_k>-1, ~ k=1,2,$ and
$w:= w_1+w_2,$ $v := v_1+v_2.$
\end{theorem}

\begin{proof}
   As noted in \ref{eq:convolution_RqRp}, the Laplace transform of the convolution $R_{\alpha ,v_1}(\lambda , t) * R_{\alpha,v_2}(\sigma, t)$ equals to $\frac{s^{v}}{(s^{\alpha} - \lambda)(s^\alpha - \sigma)}$. By assumption, $\lambda \neq \sigma$, then, using the partial fraction decomposition we have: 
   \begin{equation}
    \frac{s^v}{(s^\alpha-\lambda)(s^\alpha-\sigma)} = \frac{1}{\lambda-\sigma} \left(\frac{s^v}{s^\alpha-\lambda} - \frac{s^v}{s^\alpha-\sigma}\right).
\end{equation} 
Bringing the right hand side back to the time domain results in: 
\begin{equation}
    R_{\alpha,v_1}(\lambda,t)*R_{\alpha,v_2}(\sigma,t)  = \frac{1}{\lambda-\sigma}\big( R_{\alpha,v}(\lambda,t)-R_{\alpha,v}(\sigma,t)\big)\label{eq:Rq_Rq_redundant}
\end{equation}
which is a simple form worth calling attention to (it will be employed in the Section \ref{subsec:cwk_model} of the paper). Note, that it has a slight redundancy as the first term in the  $R$'s series expansion cancels out. Additionally, though $v_1,v_2<\alpha$, it could happen that $v=v_1+v_2 \geq \alpha$, which would violate a constraint in Definition \ref{def:R_func} requiring that $v<\alpha$ for $R_{\alpha,v}$. These issues can be resolved through the following algebraic rearrangement: for any $x,y\neq z, \frac{x}{y-z} =\frac{x}{y}  + \frac{zx/y}{y-z}$. Therefore \begin{equation}
        \frac{s^{v}}{(s^{\alpha} - \lambda) } -  \frac{s^{v}}{ (s^\alpha - \sigma)} = s^{v-\alpha} + \lambda \frac{s^{v-\alpha}}{s^\alpha-\lambda} - s^{v-\alpha} - \sigma \frac{s^{v-\alpha}}{s^\alpha-\sigma}.
    \end{equation}
     Thus \begin{equation}
          \frac{s^v}{(s^\alpha-\lambda)(s^\alpha-\sigma)} = \frac{1}{\lambda-\sigma}\left ( \lambda \frac{s^{v-\alpha}}{s^\alpha-\lambda} - \sigma  \frac{s^{v-\alpha}}{s^\alpha-\sigma}\right),
    \end{equation}
in time domain this corresponds to
\begin{equation*}
  R_{\alpha,v_1}(\lambda,t)*R_{\alpha,v_2}(\sigma,t) =  \frac{1}{\lambda-\sigma}\bigg(\lambda R_{\alpha,v-\alpha}(\lambda,t)-\sigma R_{  \alpha,v-\alpha}(\sigma,t)\bigg).
\end{equation*} 

We obtain similar results for $P$ functions,  using Equation \ref{eq:P_and_R}, \begin{align*}
    &P_{\alpha,w_1}(\lambda,t)*P_{\alpha,w_2}(\sigma,t) \\=& R_{\alpha,\alpha-w_1-1}(\lambda,t)* R_{\alpha,\alpha-w_2-1}(\sigma,t)
    \\
    =& \frac{1}{\lambda-\sigma}(\lambda R_{\alpha,\alpha-(w_1+w_2+1)-1}(\alpha,t)-\sigma R_{\alpha,\alpha-(w_1+w_2+1)-1}(\sigma,t) )
    \\
    =& \frac{1}{\lambda-\sigma}\big( \lambda P_{\alpha,w+1}(\lambda,t) - \sigma P_{\alpha,w+1}(\sigma,t) \big).
\end{align*}
Finally, to express this identity in terms of Mittag-Leffler functions we substitute $P_{\alpha,w}(\lambda,t) = t^w E_{\alpha,w+1}(\lambda t^\alpha )$, leading to 
\begin{equation}\begin{split}
    &t^{w_1}E_{\alpha,w_1+1}(\lambda t^\alpha) * t^{w_2} E_{\alpha,w_2+1}(\sigma,t)     \\ &= \frac{t^{w+1}}{\lambda-\sigma}( \lambda E_{\alpha,w+2}(\lambda t^\alpha) - \sigma E_{\alpha,w+2}(\sigma t^\alpha) )\end{split}
\end{equation}
as claimed.
\end{proof}

We extend these results to the case when $\lambda=\sigma.$
\begin{theorem} \label{thm:ConvolutionRqaRqa}
For $t>0$, $\lambda \in \C,$ and 
$\alpha>0,$  $\alpha > v_k, ~ k=1,2,$ with $v = v_1 + v_2$
    \begin{equation}
        R_{\alpha,v_1}(\lambda,t)*R_{\alpha,v_2}(\lambda,t) = \frac{t}{\alpha}R_{\alpha,v+1-\alpha}(\lambda,t) + \frac{v+1-\alpha}{\alpha}R_{\alpha,v-\alpha}(\lambda,t).
    \end{equation}
    For $w_1,w_2>-1$, with $w = w_1 + w_2$
  \begin{equation} \label{eq:repeated_lambda_P_conv}  \begin{split}
        P_{\alpha,w_1}(\lambda,t)*P_{\alpha,w_2}(\lambda,t) =& \frac{t}{\alpha}P_{\alpha,w}(\lambda,t) + \frac{\alpha-w-1}{\alpha} P_{\alpha,w+1}(\lambda,t)
    \end{split}\end{equation}

\begin{equation}\label{eq:repeated_lambda_E_conv}
    \begin{split}
        &(t^{w_1} E_{\alpha,w_1+1}(\lambda t^\alpha) )* (t^{w_2} E_{\alpha,w_2+1}(\lambda t^\alpha) )
     \\ =& \frac{1}{\alpha} t^{w+1}E_{\alpha,w+1}(\lambda t^\alpha) + \frac{\alpha-w-1}{\alpha}t^{w+1} E_{\alpha,w+2}(\lambda t^\alpha).
    \end{split}
\end{equation}
    
\end{theorem}
\begin{proof}
By the Laplace Convolution theorem (Theorem \ref{thm:ConvolutionTheorem}), the convolution of $R_{\alpha,v_1}(\lambda,t)$ with $R_{\alpha,v_2}(\lambda,t)$ is given by
 \begin{equation}
    \El\left[R_{\alpha,v_1}(\lambda,t)*R_{\alpha,v_2}(\lambda,t)\right] = \frac{s^{v_1+v_2}}{(s^\alpha-\lambda)^2} \label{eq:repeated_eigenvalue}
    \end{equation}
in Laplace space. 
To bring this to the time domain, we use the fact that derivatives in the Laplace domain correspond to multiplication by $(-t)$ in the time domain:
\begin{equation}
\El[t g(t)] = -\frac{d}{ds}\El[g](s) \label{eq:t_and_dds}
\end{equation}
(Theorem \ref{thm:Derivative_in_Laplace_space} in Appendix). Applying this to equation \ref{eq:repeated_eigenvalue} gives that
\begin{align*}
    \El[tR_{\alpha,v'}(\lambda,t)](s) 
    =& -\frac{d}{ds}\frac{s^{v'}}{s^{\alpha}-\lambda}
    \\
    =& - v'\frac{s^{v'-1}}{s^{\alpha}-\lambda} + \frac{s^{v'}}{(s^{\alpha}-\lambda)^2}\alpha s^{\alpha-1}.
\end{align*}
Letting $v' = v-\alpha-1$ ( so that $v'+\alpha-1 = v$), and rearranging, we obtain
\begin{equation}
    \frac{s^v}{(s^\alpha-\lambda)^2} = \frac{v-\alpha+1}{\alpha}\frac{s^{v-\alpha}}{s^{\alpha}-\lambda} + \frac{1}{\alpha}\El[t R_{\alpha,v-\alpha+1}(\lambda,t)](s).
\end{equation}
Returning this result to the time domain, the convolution of $R$ functions is represented as 
   \begin{equation}
   \begin{split}
    R_{\alpha,v_1}(\lambda,t)*R_{\alpha,v_2}(\lambda,t) &= \El^{-1} \left[\frac{s^v}{(s^\alpha-\lambda)^2}\right] 
    \\&= \frac{v-\alpha+1}{\alpha}R_{\alpha,v-\alpha}(\lambda,t) + \frac{t}{\alpha}R_{\alpha,v-\alpha+1}(\lambda,t).
   \end{split}
 \end{equation} 
The identities \ref{eq:repeated_lambda_P_conv} for functions $P$ are obtained from these results using Equation \ref{eq:P_and_R}. 
The remaining identities \ref{eq:repeated_lambda_E_conv} follow if we express the $P$ functions in terms of Mittag-Leffler $E$ functions.

\end{proof}

\textbf{ Remark.}
To obtain the identities corresponding to the convolution of the Mittag-Leffler functions of different types, say, for the case of an $R$ function convolved with a function $P$, first, one of the functions can be converted into the other form using Equation \ref{eq:P_and_R}.  


\subsection{\bf Generalization of Euler's identity}
As this class of functions ($R_\alpha$ or $P_\alpha$) is characterized as eigenfunctions of fractional derivatives $D^{\alpha}$, it is instructive to consider the notable examples for $\alpha = 1,2$.
Exponentials are the eigenfunctions of $D^1$, while $\sin,\cos$ or $\sinh,\cosh$ are the eigenfunctions for $D^2$, depending on the sign of the eigenvalue. Euler's formula gives an important relation between the eigenfunctions of $D^1$ and those of $D^2$. \begin{align*}
\exp(i \omega t) =& \cos(\omega t) + i \sin(\omega t)
\\
    \exp(\eta t)=& \cosh(\eta t)+ \sinh( \eta t)
\end{align*}
We can formulate these relations in terms of the Mittag-Leffler function using $E_{1,1}(z) = e^z$, $E_{2,1}(z) = \cosh(\sqrt{z})$, $E_{2,2}(z) = \frac{\sinh(\sqrt{z})}{\sqrt z}$. Hence, Euler's identity for Mittag-Leffler functions is represented as
\begin{equation}
    \begin{split} \label{eq:eulers_identity_E}
        E_{1,1}(z ) =&  E_{2,1}(z^2) + z E_{2,2}(z^2), \qquad z\in\C.
    \end{split}
\end{equation}
 We now demonstrate that this result has a further generalization for Mittag-Leffler functions (for another expression of the following relations, see \cite{LorenzoHartley2000_RFunction_relationships} and chapter 4 of \cite{lorenzo2016_fractional_Trig_book}.):

\begin{theorem} \label{thm:expansion_identity}

For $N \in \Z_{>0}$, $\alpha>0$, $w>-1$, $\beta>0$, $v<\alpha,$  
the following generalizations of Euler's relation hold:
    \begin{align}
    E_{\alpha,\beta}(z)  =&\sum_{\ell =0}^{N-1} z^\ell E_{N\alpha, \beta + \ell\alpha}(z^N) , 
    \qquad  z\in \C   
    \label{eq:E_expansion}
    \\
    P_{\alpha,w}(\lambda,t) =&\sum_{\ell=0}^{N-1} \lambda^\ell P_{N\alpha,w+\ell \alpha}(\lambda^N,t), 
    \qquad \lambda \in \C, \quad  t>0
    \label{eq:P_expansion}
    \\
    R_{\alpha,v}(\lambda,t) =& \sum_{\ell=0}^{N-1} \lambda^{N-\ell-1}R_{N\alpha,v+\alpha\ell}(\lambda^N,t) ,   \qquad \lambda \in \C, \quad  t>0 
    \label{eq:R_expansion}
\end{align}
\end{theorem}

\begin{proof}
    We initially derive \ref{eq:E_expansion} via the series representation of Mittag-Leffler functions, and then use the result to derive \ref{eq:P_expansion} and \ref{eq:R_expansion}.

    Any $j\in \Z_{\geq0}$ can be uniquely represented as $j=N k + \ell$ for $k\in\Z_{\geq 0}$, $\ell\in \{0,\cdots,N-1\}$. Hence, any unconditionally convergent series indexed by $j\in\Z_{\geq 0}$ (including the series defining the Mittag-Leffler function) can be rearranged as sums over $k\in\Z_{\geq 0}$ and  $\ell\in \{0,\cdots,N-1\}$:
\begin{equation}
    \sum _{j=0}^\infty b_j = \sum_{k=0}^{\infty} \left(b_{(Nk)} + b_{(Nk +1)} + \ldots+ b_{(Nk + N-1)} \right)  =  \sum_{\ell=0}^{N-1} \sum_{k=0}^\infty b_{(Nk+\ell)}.
\end{equation}

Applying this to the Mittag-Leffler series expansions, we obtain
\begin{align*}
    E_{\alpha,\beta}(z) =& \sum_{j=0}^\infty \frac{z^j}{\Gamma(\alpha j + \beta)}
     = \sum_{\ell=0}^{N-1} \sum_{k=0}^\infty \frac{z^{N k + \ell}}{\Gamma(\alpha (N k + \ell) + \beta)}
    \\ =& \sum_{\ell=0}^{N-1} z^\ell \sum_{k=0}^\infty \frac{z^{N k}}{\Gamma(\alpha N k + (\alpha \ell + \beta))}
    = \sum_{\ell=0}^{N-1} z^\ell E_{N\alpha,\beta+\alpha\ell}(z^N)
\end{align*}
as stated. By definition, $P_{\alpha,w}(\lambda,t) = t^w E_{\alpha,w+1}(\lambda t^\alpha)$, thus
\begin{align*}
    t^w E_{\alpha,w+1}(\lambda t^\alpha)
    =& \sum_{\ell=0}^{N-1} \lambda^\ell t^{\alpha \ell + w} E_{N\alpha,w+1+\alpha\ell}(\lambda^N t^{N\alpha})
    \\
    =& \sum_{\ell=0}^{N-1} \lambda^\ell P_{N\alpha,w+\alpha\ell}(\lambda^N,t)
\end{align*}
and $R_{\alpha,v}(\lambda,t) = t^{\alpha-v-1}E_{\alpha,\alpha-v}(\lambda t^\alpha)$ and so \begin{align*}
    t^{\alpha-v-1}E_{\alpha,\alpha-v}(\lambda t^\alpha)
    =&
    \sum_{\ell=0}^{N-1} \lambda^\ell t^{\alpha-v-1+\alpha \ell} E_{N\alpha,\alpha-v+\ell \alpha}(\lambda^N t^{N\alpha})
    \\
    =&\sum_{\ell=0}^{N-1} \lambda^\ell R_{N\alpha,v-\alpha(\ell+1-N)}(\lambda^N,t)
    \\
    =& \sum_{k=0}^{N-1} \lambda^{N-k-1} R_{N\alpha,v+\alpha k}(\lambda^N,t)
\end{align*}
where in the last sum we reversed the index $\ell=N-k-1.$
\end{proof}
Now we can see that Euler's identity \ref{eq:eulers_identity_E} is an instance of \ref{eq:E_expansion} for $N=2$. Similarly, Euler's identity written in terms of $P$ or $R$ functions takes the form \ref{eq:P_expansion} or \ref{eq:R_expansion} for $N=2.$

Additionally, there are the inverse identities representing trigonometric and hyperbolic functions using exponentials. For instance, 
\[
\frac{1}{2}(e^{+i \omega t}+ e^{- i\omega t }) =\cos(\omega t) .
\]
The next theorem generalizes these relations for Mittag-Leffler functions and $N>2.$

\begin{theorem} \label{thm:Expansion_reverse}
    Let $\zeta$ be a nontrivial $N$-th root of unity for $N \in \Z_{>0}$. For $\ell \in \Z,$ $ 0\leq \ell< N$, we have
    \begin{align} 
        \frac {1}{N}\sum_{k=0}^{N-1} (\zeta^k )^{-\ell} E_{\alpha ,\beta}(\zeta^k z)=&
        z^m E_{N\alpha,\beta+\alpha \ell}(z^N), \quad z\in\C \label{eq:E_inverse_expansion}
        \\ \frac{1}{N} \sum_{k=1}^{N} (\lambda \zeta^k)^{-\ell} P_{\alpha,w}(\lambda\zeta^k,t)  = &  P_{N\alpha,w+\alpha \ell }(\lambda^N,t), \quad \lambda \in \C, \quad t>0, \quad w>-1 
        \label{eq:P_inverse_expansion}
        \\ \frac{1}{N} \sum_{k=1}^{N} (\lambda \zeta^k)^{N-1-\ell} R_{\alpha,v}(\lambda\zeta^k,t)  =&   R_{N\alpha,v+\alpha \ell}(\lambda^N,t), \quad \lambda \in \C,\quad  t>0, \quad v<\alpha \label{eq:R_inverse_expansion}
    \end{align}

\end{theorem}

\begin{proof} 
We start with writing the left hand side of \ref{eq:E_inverse_expansion}  in the following form: 
    \begin{align*}
        &       \frac{1}{N}\sum_{\ell=0}^{N-1} (\zeta^\ell  )^{-m} E_{\alpha ,\beta}(\zeta^\ell z)
        \\
        & =   \frac{1}{N} \sum_{j=0}^\infty  \left(\sum_{\ell=0}^{N-1}(\zeta^\ell)^{j-m} \right)   \frac{z^{j }  }{\Gamma(\beta+\alpha j)}
    \end{align*}
As $\zeta$ is an $N$-th root of unity, 
    the inner sum \( \sum_{\ell=0}^{N-1} \zeta^{\ell(j - m)} \) is equal to \( N \) when \( j - m \) is a multiple of \( N \), and zero otherwise. Thus, the only nonzero terms in the outer sum correspond to \( j-m = Ni  \) for \( i \in \Z_{\geq0} \), so the sum simplifies to
    \begin{align*}
        &= \frac{N}{N}  \sum_{i=0}^\infty \frac{z^{Ni+m}}{\Gamma(\beta+m \alpha +N\alpha i)}
        \\&= z^mE_{N\alpha,\beta+m\alpha}(z^N)
    \end{align*}
as claimed. The identities for $P$ functions are obtained by representing them in terms of Mittag-Leffler functions.
\begin{align*}
\frac{1}{N} \sum_{k=1}^{N} (\lambda \zeta^k)^{-m} P_{\alpha,w}(\lambda\zeta^k,t)  
        =&  \frac{1}{N} \sum_{k=1}^{N} (\lambda \zeta^k)^{-m} t^{w} E_{\alpha,w+1}( \lambda \zeta^k t^\alpha)
    \\  =& \lambda^{-m}t^w (\lambda t^\alpha)^m E_{N\alpha,w+m\alpha+1}(\lambda^N t^{N\alpha})
    \\  =& P_{N\alpha,w+m\alpha}(\lambda^N,t)
\end{align*}
Using this result and \ref{eq:P_and_R}, the identities for $R$ functions are the following: 
\begin{align*}
    &\frac{1}{N} \sum_{k=1}^{N} (\lambda \zeta^k)^{m+1-N} R_{\alpha,v}(\lambda\zeta^k,t)
         \\     =&  \frac{1}{N} \sum_{k=1}^{N} (\lambda \zeta^k)^{m+1-N} P_{\alpha,\alpha -  v - 1}(\lambda\zeta^k,t)
         \\     =&  P_{N\alpha,\alpha-v-1-\alpha(m+1-N)}(\lambda^N,t) = P_{N\alpha, N\alpha -v -\alpha m -1}(\lambda^N,t)
         \\     =& R_{N\alpha,v+\alpha m}(\lambda^N,t)
\end{align*}
as claimed.
\end{proof}

 \subsection{\bf Convolution to sum identities 
 when $\frac{\alpha}{\beta} = \frac{n}{m}$}

An immediate application of Theorem \ref{thm:expansion_identity} is that $P/R$ functions of orders $\alpha$ and $\beta$ can  be expressed as sums of $P/R$ functions of orders $m\alpha$ and $n\beta$ respectively. If $\alpha$ and $\beta $ are related so that $m\alpha=n\beta$,  then a convolution of the order-$\alpha$ and order-$\beta$ functions can be expressed as a sum of convolutions of $P/R$ functions whose order is the same, and can be evaluated through Theorem \ref{thm:ConvolutionRqRq}.
One may expect the resulting sum to have $2nm$ terms (where $m$, $n$ are the numbers of terms in each expansion). However, the next theorem shows that it can be reduced to a sum of $n+m$ $P/R$ functions.

\begin{theorem}
 \label{thm:Main_Convolution_theorem}
    Let $\frac{\alpha}{\beta} = \frac{n}{m}$ where $n,m$ are positive coprime integers (the quotient $\frac{n}{m}$ is irreducible fraction), 
    $ \lambda^n \neq \sigma^m$, $w_1,w_2>-1,$ and $w=w_1+w_2$. Then, the convolution of $P_{\alpha,w_1}(\lambda,t)$ and $P_{\beta,w_2}(\sigma,t)$ can be written as a sum of $n+m$ $P$ functions:
    \begin{equation}
    \begin{split}
    &P_{\alpha,w_1}(\lambda,t)*P_{\beta,w_2}(\sigma,t) \\
    =&
       \sum_{j=0}^{n-1} \frac{\sigma^j \lambda^m}{\lambda^m-\sigma^n} P_{\alpha,w+1+\beta j}(\lambda,t) + \sum_{k=0}^{m-1} \frac{\lambda^k \sigma^n}{\sigma^n -\lambda^m} P_{\beta,w+1+\alpha k}(\sigma,t)
       \end{split}
    \end{equation}
    Expressed in terms of Mittag-Leffler functions, the convolution is 
    \begin{equation}\begin{split}
    &t^{w_1}E_{\alpha,w_1+1}(\lambda t^\alpha) * t^{w_2} E_{\beta,w_2+1}(\sigma t^\beta)
    \\=&
            \sum_{j=0}^{n-1} \frac{\sigma^j \lambda^mt^{w+1+\beta j}}{\lambda^m-\sigma^n} E_{\alpha,w+2+\beta j}(\lambda t^\alpha) + \sum_{k=0}^{m-1} \frac{\lambda^k \sigma^n t^{w+1+\alpha k}}{\sigma^n-\lambda^m} E_{\beta,w+2+\alpha k}(\sigma t^\beta).\end{split}
    \end{equation}
For $R$ functions, with $v_1<\alpha,v_2<\beta$, $v=v_1+v_2 $, the convolution can be represented as
\begin{equation}\begin{split}
&R_{\alpha,v_1}(\lambda,t)* R_{\beta ,v_2}(\sigma,t)
 \\=&   \sum_{j=0}^{n-1}\frac{\sigma^j \lambda^m}{\lambda^m-\sigma^n}R_{\alpha,v-\beta(j+1)}(\lambda,t) + \sum_{k=0}^{m-1} \frac{\lambda^k \sigma^n}{\sigma^n-\lambda^m}R_{\beta,v - \alpha(k+1)}(\sigma,t).\end{split}
\end{equation}

\end{theorem}

\begin{proof}
     Since $m\alpha = n\beta$, we expand $P_\alpha$ as a sum of $P_{m\alpha}$ functions (\ref{eq:P_expansion}) and $P_\beta$ as a sum of $P_{n\beta}$ functions.
\begin{align*}
    &P_{\alpha,w_1}(\lambda,t)*P_{\beta,w_2}(\sigma,t) 
   \\ =& \sum_{k=0}^{m-1} \sum_{j=0}^{n-1} \lambda^k \sigma^j P_{m\alpha,w_1+\alpha k}(\lambda^m,t)* P_{n\beta,w_2+\beta j}(\sigma^n,t)
    \\
    =& \sum_{k=0}^{m-1} \sum_{j=0}^{n-1} \lambda^k \sigma^j  \frac{1}{\lambda^m-\sigma^n}\bigg( \lambda^mP_{m\alpha,w+1+\alpha k+\beta j }(\lambda^m,t)- \sigma^n P_{n\beta, w+1+\alpha k+\beta j}(\sigma^n,t)\bigg) 
\end{align*}
    This is a linear combination of $P$ functions, but with $2nm$ terms. However, applying the expansion identity \ref{eq:P_expansion} to gather the terms in the sum, we obtain:
    \begin{equation}
        \sum_{k=0}^{m-1} \lambda^k P_{m\alpha,w+1+\alpha k+\beta j }(\lambda^m,t) =  P_{\alpha,w+1+\beta j}(\lambda,t)
    \end{equation}
    Likewise steps are applied to the other term. Together, these simplifications give 
    \begin{equation}
        \frac{1}{\lambda^m-\sigma^n}\bigg( \sum_{j=0}^{n-1}  \lambda^m \sigma^jP_{\alpha,w+1+\beta j}(\lambda,t) - \sum_{k=0}^{m-1} \sigma^n 
        \lambda^{k} P_{\beta,w+1+\alpha k}(\sigma,t) \bigg)
    \end{equation}
which rearranges to the stated form.

Converting these to $R$ identities (which could as well be done via expansion and convolution of $R$ functions), we have the identities for convolutions of $R$ functions:
\begin{align*}
   & R_{\alpha,v_1}(\lambda,t)*R_{\beta,v_2}(\sigma,t) 
    \\
    =&P_{\alpha,\alpha-v_1- 1}(\lambda,t)*P_{\beta,\beta-v_2-1}(\sigma,t)
    \\
    =& \sum_{j=0}^{n-1}\frac{\sigma^j \lambda^m}{\lambda^m-\sigma^n}P_{\alpha,\alpha + \beta -v -1 + \beta j}(\lambda,t) + \sum_{k=0}^{m-1} \frac{\lambda^k \sigma^n}{\sigma^n-\lambda^m}P_{\beta,\beta+\alpha-v -1 + \alpha k}(\sigma,t)
    \\
    =& \sum_{j=0}^{n-1}\frac{\sigma^j \lambda^m}{\lambda^m-\sigma^n}R_{\alpha,v-\beta(j+1)}(\lambda,t) + \sum_{k=0}^{m-1} \frac{\lambda^k \sigma^n}{\sigma^n-\lambda^m}R_{\beta,v - \alpha(k+1)}(\sigma,t)
\end{align*}

\end{proof}

To convolve a $P_{\alpha,w}(\lambda,t)$ with an $R_{\beta,v}(\sigma,t)$, one needs to convert them both to either $R$ or $P$ form using Equation \ref{eq:P_and_R}. 

The inverse expansion identity, Theorem \ref{thm:Expansion_reverse}, can also be used to give a family of equivalent convolution-to-sum identities by expanding both in order $\frac{\alpha}{n} = \frac{\beta}{m}$ (not just one, because we can take $\ell$ to be nonzero in the secondary parameter for each expansion). For the sake of brevity, we only demonstrate the un-simplified version of this identity for $P$ functions.

 \begin{theorem}
     Let parameters $\alpha,\beta,n,m,\lambda,\sigma, w_1,w_2,w,v_1,v_2,v$  be as in Theorem \ref{thm:Main_Convolution_theorem}.  
     Let  $\zeta, \xi$ be $n$-th and $m$-th nontrivial roots of unity respectively, and $\ell_n ,\ell_m \in \Z_{\geq0}$ with $\ell_n<n,\ell_m<m.$ 
     Then the convolution of $P_{\alpha,w_1}(\lambda,t)$ and $P_{\beta,w_2}(\sigma,t)$ can be written as a sum of $n+m$ $P$ functions as
     \begin{equation}
         \begin{split}
             &P_{\alpha,w_1}(\lambda,t)*P_{\beta,w_2}(\sigma,t) 
             \\=& \sum_{j=0}^{n-1}   a_j P_{\frac{\alpha}{n},w+1 - (\ell_n + \ell_m) \frac{\alpha}{n} }(\zeta^j \lambda^{\frac{1}{n}} ) 
        +  \sum_{k=0}^{m-1} b_k P_{\frac{\beta}{m},w+1 - (\ell_n + \ell_m) \frac{\beta}{m} }(\xi^k \sigma^{\frac{1}{m}} ) 
         \end{split}
     \end{equation}
     where $a_j$ and $b_k$  $\in \C$ are coefficients which depend on $n,m,\lambda,\sigma,\ell_m,$ and $\ell_n$.
 \end{theorem}
\begin{proof}
    Because $\frac{\alpha}{\beta} = \frac{n}{m},$ $\frac{\alpha}{n} =\frac{\beta}{m} $, we represent the convolution of the $P$ functions as:  
    \begin{align*}
        &P_{\alpha,w_1}(\lambda,t)*P_{\beta,w_2}(\sigma,t) 
        \\=&\frac{1}{nm}\sum_{j=0}^{n-1}\sum_{k=0}^{m-1} 
        (\zeta^j \lambda^{\frac{1}{n}} )^{-\ell_n}(\xi^k \sigma^{\frac 1m})^{-\ell_m} P_{\frac{\alpha}{n},w_1 -\ell_n \frac{\alpha}{n}}(\zeta^j \lambda^{\frac{1}{n}} )* 
        P_{\frac{\beta}{m},w_2 -\ell_m \frac{\beta}{m}}(\xi^k \sigma^{\frac{1}{m}} )
        \\=&\frac{1}{nm}\sum_{j=0}^{n-1}\sum_{k=0}^{m-1} 
         \frac{
       \zeta^j \lambda^{\frac 1n} P_{\frac{\alpha}{n},w+1 - (\ell_n + \ell_m) \frac{\alpha}{n} }(\zeta^j \lambda^{\frac{1}{n}} ) - \xi^k \sigma^{\frac 1m}
        P_{\frac{\beta}{m},w+1 - (\ell_n + \ell_m) \frac{\beta}{m}}(\xi^k \sigma^{\frac{1}{m}} )
        }   {
        (\zeta^j \lambda^{\frac{1}{n}} )^{\ell_n}(\xi^k \sigma^{\frac 1m})^{\ell_m} (\zeta^j \lambda^{\frac{1}{n}}-\xi^k \sigma^{\frac{1}{m}} )
        }
        \\=&
        \frac{1}{nm} \sum_{j=0}^{n-1} P_{\frac{\alpha}{n},w+1 - (\ell_n + \ell_m) \frac{\alpha}{n} }(\zeta^j \lambda^{\frac{1}{n}} ) 
            \left(\sum_{k=0}^{m-1} 
                \frac{1}{
                (\zeta^j \lambda^{\frac{1}{n}} )^{\ell_n-1}(\xi^k \sigma^{
                \frac 1m})^{\ell_m} (\zeta^j \lambda^{\frac{1}{n}}-\xi^k \sigma^{\frac{1}{m}} )} \right)
        \\&+ \frac{1}{nm} \sum_{k=0}^{m-1} P_{\frac{\beta}{m},w+1 - (\ell_n + \ell_m) \frac{\beta}{m} }(\xi^k \sigma^{\frac{1}{m}} ) 
            \left(\sum_{j=0}^{n-1} 
                \frac{1}{
                (\zeta^j \lambda^{\frac{1}{n}} )^{\ell_n}(\xi^k \sigma^{
                \frac 1m})^{\ell_m-1} (\xi^k \sigma^{\frac{1}{m}}  - \zeta^j \lambda^{\frac{1}{n}})} \right)
        \\=&  \sum_{j=0}^{n-1}   a_j P_{\frac{\alpha}{n},w+1 - (\ell_n + \ell_m) \frac{\alpha}{n} }(\zeta^j \lambda^{\frac{1}{n}} ) 
        +  \sum_{k=0}^{m-1} b_k P_{\frac{\beta}{m},w+1 - (\ell_n + \ell_m) \frac{\beta}{m} }(\xi^k \sigma^{\frac{1}{m}} ) 
    \end{align*}
as claimed, where the coefficients $ a_j$ and  $ b_k$ have the following form:
\begin{equation}
    \begin{split}
        a_j =& \frac{1}{nm}\sum_{k=0}^{m-1} 
                \frac{1}{
                (\zeta^j \lambda^{\frac{1}{n}} )^{\ell_n-1}(\xi^k \sigma^{
                \frac 1m})^{\ell_m} (\zeta^j \lambda^{\frac{1}{n}}-\xi^k \sigma^{\frac{1}{m}} )}
    \\  b_k =& \frac{1}{nm}\sum_{j=0}^{n-1} 
                \frac{1}{
                (\zeta^j \lambda^{\frac{1}{n}} )^{\ell_n}(\xi^k \sigma^{
                \frac 1m})^{\ell_m-1} (\xi^k \sigma^{\frac{1}{m}}  - \zeta^j \lambda^{\frac{1}{n}})}.
    \end{split}
\end{equation}

\end{proof}

In the next section, we present examples of employing the derived fundamental convolution identities to solution of nonhomogeneous equations with fractional derivative operators.  

\section{Applications and examples} \label{sec:4}

\setcounter{section}{4} \setcounter{equation}{0} 
 
\subsection{\bf Nonhomogeneous equatiions}

The developed convolution-to-sum identities being fundamental in the theory of Mittag-Leffler functions, are instrumental in constructing solutions of nonhomogeneous fractional-differential equations with a forcing term. Consider the modification of equation \ref{eq:eigenfunction_equation} with a source term:
\begin{equation}
    D^\alpha f(t) - \lambda f(t) = g(t) \label{eq:inhomogenous_fundamental_eqn}.
\end{equation}
The left hand side consists of a linear operator, $(D^\alpha-\lambda)$, acting on $f$, and as such the solution to this equation consists of a homogeneous solution (determined by the initial data) plus a nonhomogeneous solution determined by the inhomogeneous term. We have already derived the homogeneous solution, it is given by \ref{eq:eigenfunction_solutions}. Let us construct the nonhomogeneous solution due to the source term in the right hand side. Reviewing \ref{eq:LaplaceOfCaputo},\ref{eq:LaplaceOfRL} reveals that, modulo initial data, the RL and Caputo derivatives are the same in Laplace space: $\El[D^\alpha f] = s^{\alpha} \El[f].$ Therefore, the inhomogeneous solution, $f_{in}$ does not depend on which convention is used. Taking  \ref{eq:inhomogenous_fundamental_eqn} into Laplace space:
\begin{align*}
    s^{\alpha}\El[f_{in}] - \lambda \El[f_{in}] =& \El[g]
    \\
    \El[f_{in}] = \frac{1}{s^\alpha-\lambda} \El[g].
\end{align*}
By the Convolution Theorem (\ref{thm:ConvolutionTheorem}), \begin{equation}
    f_{in} = R_{\alpha,0}(\lambda,t)*g. \label{eq:inhomogenous_solution}
\end{equation}
Notice that function $P/R$ encapsulate very general behavior. 
For example, exponential and trigonometric functions are particular case of them:
\begin{align*}
    \exp(at) =& R_{1,0}(a,t) = P_{1,0}(a,t)
    \\
    \cos(\omega t) =& R_{2,1}(-\omega^2,t) = P_{2,0}(-\omega^2,t)
    \\
    \sin(\omega t) =&  \omega  R_{2,0}(-\omega^2,t) = \omega P_{2,1}(-\omega^2,t).
\end{align*}
Hence, 
provided $\alpha$ is rational, Theorem \ref{thm:Main_Convolution_theorem} provides an explicit sum form for the convolution of $R_{\alpha,0}(\lambda)$ with any linear combination of the exponential, trigonometric, and hyperbolic functions (in addition to any $P$ or $R$ functions of rational order). Representation of the solution as a sum of Mittag-Leffler functions instead of their convolution, significantly reduces required computations and results in efficient numerical methods of solution of nonhomogeneous fractional differential equations.

\subsection{\bf Inhomogeneous subdiffusion}

As an explicit example, consider a time-fractional subdiffusion model, described in \cite{Mainardi_Gorenflo_2002_fractional_diffusion_random_walk,Mainardi_1994_IVP_fractional_diffusionwaveeq,Mainardi1996_Fractional_Relaxation_oscillation}, more generally in \cite{Metzler2000_Random_Walk_Anomalous_Diff,Saichev_Zaslavsky_1997_Fractional_kinetic_eqs,Henry_Langlands_Straka_2010_intro_fractional_diffusion} with a physical realization in \cite{Riechers2024_Cluster_dynamics_Fractional_diffusion_mettalic_glass}:
\begin{equation} \label{eq:subdiff_model}
    \caputo[]{\alpha}{t} u(x,t) - a \Delta_x u(x,t) = g(x,t), \,\quad u(x,t=0) = 0
\end{equation}
for $0 < \alpha < 1$, in $\Omega\times [0,\infty) $ in a bounded, smooth, simply connected domain $\Omega \subset \R^n$, with $u = 0$ on $\partial \Omega$. The eigenfunctions of the Laplacian $\{\phi_k(x)\}$ satisfy $\Delta_x \phi_k = -k^2 \phi_k$. The functions $\{\phi_k\}$ can be chosen to form a complete orthonormal basis in $L^2(\Omega)$; the function $u$ in this basis has the form: 
\begin{equation} \label{eq:modal_expansion_of_u}
u(x,t) = \sum_k \tilde u_k(t) \phi_k(x),\, \quad \tilde u_k(t) = \int_\Omega u(x,t) \phi_k (x) d x .
\end{equation}
A similar representation in terms of functions $\phi_k(x)$ holds for the function $g$ with the coefficients $\tilde g_k(t)$.
Then, the fractional differential equation can be written for each mode $\tilde u_k$ as:
\begin{equation}
    D^\alpha_t \tilde u_k(t) + a k^2 \tilde u_k(t) = \tilde g_k(t).
\end{equation}
This equation is a particular case of \ref{eq:inhomogenous_fundamental_eqn},
and thus the solution for each mode is given by (see \ref{eq:inhomogenous_solution}) \begin{equation}
  \begin{split}
        \tilde u_k(t) =  R_{\alpha,0}(-a k^2,t) *\tilde g_k.
  \end{split}
\end{equation}
The initial conditions are specified to be zero in \ref{eq:subdiff_model}, but had they been nonzero, there would be an additional (homogeneous) term in the solution for each mode, given by Theorem \ref{thm:solution_to_eigenproblem}. Suppose $\alpha = 1/2$, and that the source term for each mode is given by $\tilde g_k(t) = c_ke^{-r_kt}$. Then $\tilde u_k$ is given by \begin{equation}
    \tilde u_{k}(t) = R_{1/2,0}(-ak^2,t)*\exp(-r_k t) = R_{1/2,0}(-ak^2,t)* R_{1,0}(-r_k,t).
\end{equation}
This is a convolution of $R$ functions so applying Theorem \ref{thm:Main_Convolution_theorem} with $ n=1,m=2$, $\tilde u_k$ can be represented as a sum of only three $R$ functions:
\begin{equation}
    \tilde u_{k}(t) = \frac{c_k}{a^2k^4 +r} \bigg( a^2 k^4 R_{1/2,-1}(-ak^2,t) + r_k R_{1,-1/2}(-r_k,t) - r_k ak^2 R_{1,-1}(-r_k,t) \bigg).\label{eq:inhom_mode_subdiff}
\end{equation} 
Alternatively, $\tilde u_k$ can be expressed in terms of $P$ functions, because $R_{1/2,0}=P_{1/2,-1/2}.$ This representation also has only three terms:
\begin{equation}\begin{split}
    \tilde u_{k}(t) =& c_k P_{1/2,-1/2}(-ak^2)*P_{1,0}(-r_k)\\=& \frac{c_k}{a^2k^4+r_k}\bigg( a^2k^4 P_{1/2,1/2}(-ak^2,t)  + r_k P_{1,1/2}(-r_k,t) -k^2r_k P_{1,1}(-r_k,t)\bigg).\end{split}
\end{equation}
$\tilde u_k$ for $k=1,r=1.5$ is displayed in the left panel Figure \ref{fig:inhomogeneous_solutions} with the exponential forcing function $e^{-1.5 t}$.

For another example, consider \ref{eq:subdiff_model} for $\alpha=3/4$ with sinusoidal forcing term $\tilde g_k = c_k\cos(\omega_k t) = R_{2,1}(-\omega_k^2 ,t) = P_{2,0}(-\omega_k^2,t).$ The inhomogeneous solution for each mode is thus given by $R_{3/4,0}(-ak^2)*P_{2,0}(-\omega_k^2)$.  In this case, $\frac{3/4}{2} = \frac{n}{m}$, so $n=3$, $m=8$, and $R_{3/4,0} = P_{3/4,-1/4}$.  Using  Theorem \ref{thm:Main_Convolution_theorem}, $\tilde u_k$ can be represented as
\begin{equation}\begin{split}
    &\tilde u_{k}(t) = c_k P_{3/4,-1/4}(-ak^2,t)*P_{2,0}(-\omega_k^2,t)
    \\=&c_k\sum_{j=0}^2 \frac{(-\omega_k^2)^j a^8k^{16}}{a^8k^{16} -\omega_k^6}P_{3/4,2j-1/4}(-ak^2,t) + c_k\sum_{\ell=0}^{7}\frac{(-ak^2)^\ell \omega_k^6}{\omega_k^6-a^8k^{16}}P_{2,(3j-1)/4}(-\omega_k^2,t). 
    \end{split}\label{eq:sol_to_inhom_diffusion_with_cosine}
\end{equation}
The right panel of Figure \ref{fig:inhomogeneous_solutions} displays this solution along with the sinusoidal forcing term for $k=2$ and $\omega_k = 2$.

\begin{figure}[t] 
  \begin{center}
    \begin{minipage}[b]{0.48\textwidth}
      \centering
      \includegraphics[width=\textwidth]{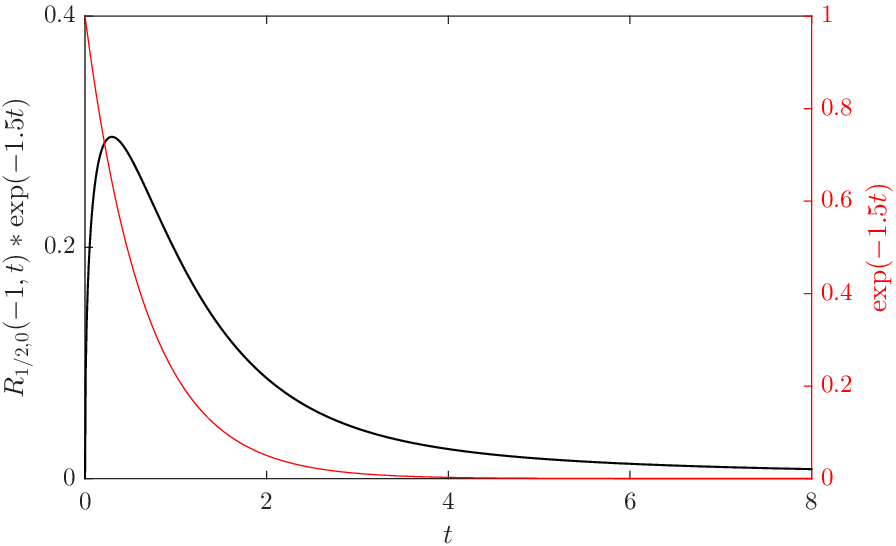}
    \end{minipage}
    \hfill
    \begin{minipage}[b]{0.49\textwidth}
      \centering
      \includegraphics[width=\textwidth]{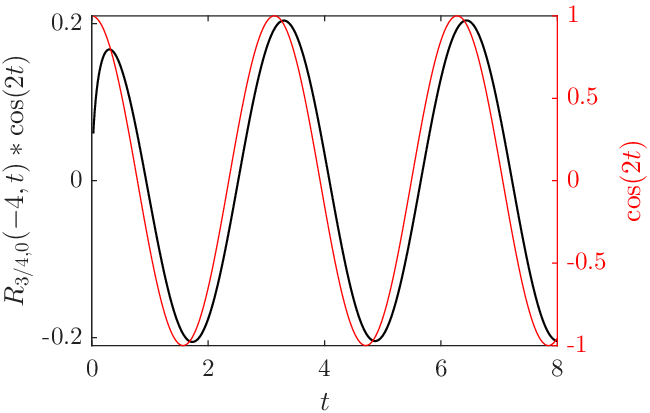}
    \end{minipage}

    \caption{Inhomogeneous solutions to $D^{\alpha} f +k^2 f = g$ together with the forcing term $g$. The left panel shows solution corresponding to $\alpha = 1/2$, $g(t)=e^{-1.5 t}$, and $k=1$. The right panel presents the solution for $\alpha = 3/4$, $g(t)=\cos(2 t)$, and $k=2$. Note that on the left, though the source term decays exponentially, the inhomogeneous solution has a 'long tail' characteristic of subdiffusion \cite{Mainardi1996_Fractional_Relaxation_oscillation}.}
    \label{fig:inhomogeneous_solutions}
  \end{center}
\end{figure}

\subsection{\bf  Attenuated wave model}
\label{subsec:cwk_model}
Michele Caputo introduced the "Caputo" derivative in \cite{Caputo1967_Linear_models_Of_Dissipation} to construct a wave model with  almost frequency independent dissipation. This model (see equation 8 in \cite{Caputo1967_Linear_models_Of_Dissipation}, as well as subsequent works \cite{Wismer2006_Acoustic_pulses_Inhom_media_powerlaw_attenuation}, \cite{Holm2013_Deriving_fractional_Wave_eq_etc}, \cite{Kaltenbacher2021_Inverse_problems_for_Wave_eqs_with_frac_attenuation}) is referred to as the fractional Kelvin-Voigt or the Caputo-Wismer-Kelvin model:
\begin{equation}
    D_t^2 u- \eta \Delta_x \caputo[]{\alpha}{t} u-c^2 \Delta_x u  = g \label{eq:cwk_equation}.
\end{equation}
Here  $0 < \alpha < 1$ and the equation is considered in $\Omega\times [0,\infty) $, where $\Omega \subset \R^n$ 
is a bounded, smooth, simply connected domain with sufficient constraints on $u\vert_{\partial \Omega}$ for the Laplacian $\Delta_x$ to have a discrete spectrum. We include this model to demonstrate another use-case of convolution-to-sum identities, and illustrate an effective approximation that works even when the order of the fractional derivative $\alpha$ is irrational.

After expanding the solution $u$ of \ref{eq:cwk_equation} in terms of the orthonormal eigenfunctions $\phi_k(x)$ of $\Delta_x$ as given in \ref{eq:modal_expansion_of_u}, the equation for each mode $u_k(t)$ is:
\begin{equation}
    D^2_t \tilde u_k(t) +(\eta k^2 )\caputo[]{\alpha}{t} \tilde u_k(t) + c^2 k^2 \tilde u_k(t)  = \tilde g_k(t),
\end{equation}
where $\tilde g_k(t)$ are the coefficients of the expansion of $g(x,t)$ in terms of eigenfunctions of the Laplacian $\phi_k(x)$.
This equation has time derivatives in two terms unlike the equations previously considered. 
Just as when solving the eigenfunction equation, we transform this equation into the Laplace domain,
taking $\tilde U_k(s))  =\El[\tilde u_k(t)](s), \tilde G_k(s) = \El[\tilde g_k(t)](s)$ and using  \ref{eq:LaplaceOfCaputo}: 
\begin{equation}
    s^2 \tilde U_k(s) - s \tilde u_k(0) - \frac{\partial\tilde u_k}{\partial t}\bigg|_{t=0} + \eta k^2 (s^\alpha \tilde U_k(s)- s^{\alpha-1} \tilde u_k(0)) + c^2 k^2 \tilde U_k(s) = \tilde G_k(s).
\end{equation}
Solving for $\tilde U_k(s)$ yields 
\begin{equation} \label{eq:laplace_unsimplified_sol_cwk}
     \tilde U_k(s)  =   \frac{1}{  s^2 + \eta k^2 s^\alpha + c^2k^2 }\left(s \tilde u_k(0)   +   \eta k^2  s^{\alpha-1} \tilde u_k(0)    +      \frac{\partial\tilde u_k}{\partial t} \bigg|_{t=0} + \tilde G_k(s) \right).
\end{equation}

This form is analogous to those considered earlier (such as $\frac{s^v}{s^\alpha-\lambda}$), but the additional term in the denominator prevents us from representing this as a geometric series.

Suppose the order of the derivative is rational,  $\alpha = \frac mn$ (with $m<n$ as $0<\alpha<1$). Then the denominator can be expressed as a polynomial in $z =s^{1/n}$. Specifically,
\begin{equation}
  p(z) = z^{2n} +\eta k^2 z^m + c^2k^2.
\end{equation}

By assumption the Laplacian has a discrete spectrum, so this order-$2n$ polynomial can be factored for each value of $k^2$ via the fundamental theorem of algebra:
\begin{equation}
    z^{2n} +\eta k^2 z^m + c^2k^2 = \prod_{j=1}^{2n} (z-\lambda_{kj})
\end{equation}
(the roots have been labeled with index $k$ as they are mode-dependent). Now the factor $\frac{1}{s^2+\eta k^2s^\alpha + c^2 k^2}$ is represented as
\begin{equation}
  \frac{1}{s^2+\eta k^2s^\alpha + c^2 k^2} =   \prod_{j=1}^{2n}\frac{1}{(s^{1/n}-\lambda_{kj})},
\end{equation}
 and may be equivalently treated through partial fraction decomposition or as a repeated convolution of $R_{1/n}$ functions in time.

The representation of the convolution between $R$ functions given in \ref{eq:Rq_Rq_redundant}, can be used here to demonstrate the equivalence to partial fraction decomposition in the Laplace domain:
\begin{equation} \label{eq:redundant_r_conv_returned}
    R_{\alpha,v_1}(\lambda,t)* R_{\alpha,v_2}(\sigma,t) = \frac{1}{\lambda-\sigma} R_{\alpha,v}(\lambda,t)  +\frac{1}{\sigma-\lambda} R_{\alpha,v}(\sigma,t) .
\end{equation}
Assuming $p(z)$ has no repeated roots and applying \ref{eq:redundant_r_conv_returned} $2n$ times yields that
\begin{equation}
    \Conv_{j=1}^{2n} R_{1/n,v_j}(\lambda_{kj},t)= \sum_{j=1}^{2n}  \left(\prod_{\ell\neq j} \frac{1}{\lambda_{kj}-\lambda_{k\ell}} \right)R_{1/n,v}(\lambda_{kj},t),
\end{equation}
where $\Conv$ is used to represent repeated convolution. This form is identical to the result derived via partial fraction decomposition:
\begin{equation}
    \prod_{j=1}^{2n}\frac{s^v}{(s^{1/n}-\lambda_{kj})}= \sum_{j=1}^{2n} \frac{s^v}{s^{1/n}-\lambda_{kj}}\prod_{\ell\neq j} \frac{1}{\lambda_{kj} - \lambda_{k\ell}} .
\end{equation}

Hence, the solution for each mode is given by \begin{equation} \begin{split}
    \tilde u_k(t) =& \sum_{j=1}^{2n}  \left(\prod_{\ell\neq j} \frac{1}{\lambda_{kj}-\lambda_{k\ell}} \right)
    \bigg(\tilde u_k(0)R_{1/n,1}(\lambda_{kj},t) + \\&\eta k^2 \tilde u_k(0) R_{1/n,\alpha-1}(\lambda_{kj},t) + \frac{d\tilde u_k}{dt}\bigg|_{t=0} R_{1/n,0}(\lambda_{kj},t) + \tilde g_k(t) * R_{1/n,0}(\lambda_{kj},t) \bigg)
\end{split} \label{eq:time_domain_solution_of_cwk_for_modek}
\end{equation}
except in the case of repeated roots, when the results of Theorem \ref{thm:ConvolutionRqaRqa} can be used.

\begin{figure}[H]
    \centering
    \includegraphics[scale=.375]{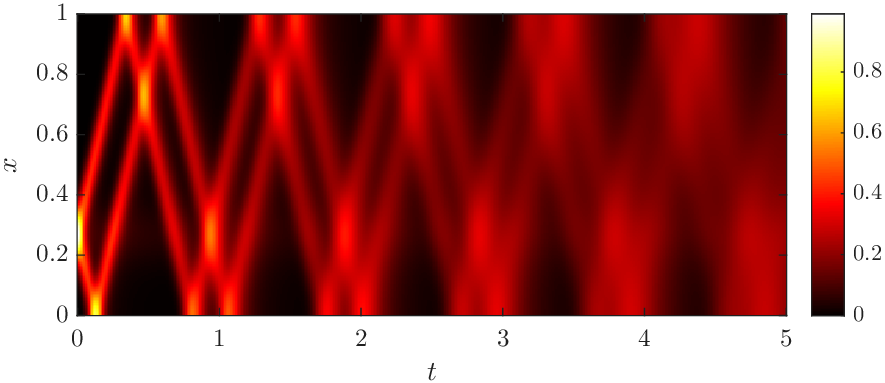} \includegraphics[scale=.375]{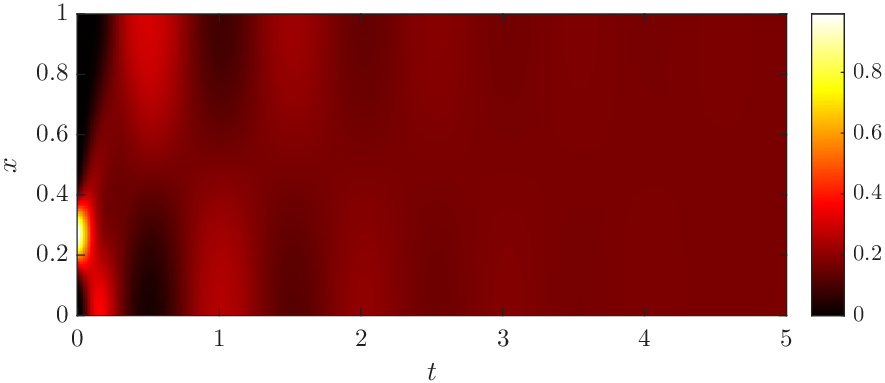}
    \caption{A comparison of solutions to \ref{eq:cwk_equation} in one spatial dimension with $\alpha=1/3$ on the left and $\alpha=1$ on the right. For both, $\eta = 0.2$, $c=2$, initial conditions are given by $u(x,t=0) = e^{-(x-0.27)^2/0.1^2) },\frac{\partial u }{\partial t}\big|_{t=0} = 0$, with  Neumann boundary conditions at $x=0,1$. Qualitatively, for $\alpha <1$ the wave stays more coherent as it decays--displaying the 'approximately frequency-independent damping' described by Caputo.}
    \label{fig:cwk_neumann}
\end{figure}

The full solution in time and space is given by summing over all modes as in \ref{eq:modal_expansion_of_u}; an example is displayed in Figure \ref{fig:cwk_neumann}. 
The developed method gives an analytic solution to \ref{eq:cwk_equation} in terms of Mittag-Leffler functions, except with the single numerical step of factoring the polynomial.
However, this method is not directly applicable if $\alpha$ is irrational (unless we approximate it by a rational number) and the relationship between $\alpha,$  $\eta$, $k$, and the decay rates of the solution is not clear.
To show how these challenges can be bypassed for this model, we follow Caputo in \cite{Caputo1967_Linear_models_Of_Dissipation}, gaining insight into solution behavior by approximating the roots of the denominator for small $\eta$.

   The denominator of equation \ref{eq:laplace_unsimplified_sol_cwk} is given by \[
   M_k(s) = s^2+\eta k^2 s^\alpha + c^2 k^2.\]
    Using an asymptotic expansion in $\eta\ll c^2$, with $s_* = s_0 + \eta s_1 + O(\eta^2)$, such that $M_k(s_*) = 0,$ yields that 
\begin{equation}
    \begin{split}
        s_0 =& \pm ick
        \\
        s_1 =& -\frac{ 1}{2}s_0^{\alpha-1} k^2 = -\frac12 c^{\alpha-1} k^{\alpha+1} (\pm i)^{\alpha-1}
    \end{split}.
\end{equation}
From here, the approximation to the root $s_*$ is
\begin{align*}
    s_* &= \pm i \omega_k - r_k + O(\eta^2)
\end{align*}
where
\begin{equation}
    \begin{split}
        r_k =& \eta\frac12  c^{\alpha-1} k^{\alpha+1} \sin(\frac{\pi\alpha}{2})
       \\ \omega_k =& ck + \eta \frac12  c^{\alpha-1} k^{\alpha+1}  \cos(\frac{\pi\alpha }{2}) .
    \end{split} \label{eq:rk_wk_definition}
\end{equation}
Hence,  $M_k(s)$ can be approximated as $(s+r_k+i\omega_k)(s+r_k-i\omega_k)$. Consider the first term in \ref{eq:laplace_unsimplified_sol_cwk}, assuming $\tilde u_k(0) = 1$, $
    \chi_k(s) =\frac{s}{M_k(s)}.$ 
This term of the solution can be approximated as 
\begin{equation} \label{eq:m_k approximaiton} \chi_k(s) =  \frac{s}{M_k(s)} \approx \frac{s}{(s+r_k+i\omega_k)(s+r_k-i\omega_k)} .
\end{equation}
The corresponding inverse Laplace transform of the right-hand side of \ref{eq:m_k approximaiton} is 
\begin{equation} \label{eq:chi_t_approximation}
    \chi_k(t) = \El^{-1}[\chi_k(s)](t) \approx \cos(\omega_kt) e^{-r_k t} - \frac{r_k}{\omega_k} \sin(\omega_kt) e^{-r_k t}.
\end{equation} The exact functions $\chi_k(t)$ for $k=1,2,4,8$, and their approximations as given in the right-hand side of \ref{eq:chi_t_approximation}, are shown in Figure \ref{fig:cwk_modes}. The exact representations in time are given by the first term in \ref{eq:time_domain_solution_of_cwk_for_modek}, which is obtained through the convolution methods for Mittag-Leffler functions developed in this paper.

Similarly, the exact driven solution, given by the last term of \ref{eq:time_domain_solution_of_cwk_for_modek}, can be approximated in the Laplace domain as 
\begin{equation}
    \frac{\tilde G_k(s)}{M_k(s)} \approx \frac{\tilde G_k(s)}{(s+r_k+i\omega_k)(s+r_k-i\omega_k)}.
\end{equation}
The inverse Laplace transform of $\frac{1}{(s+r_k+i\omega_k)(s+r_k-i\omega_k)}$ is $\frac{1}{\omega_k} \sin(\omega_k t) e^{-r_k t}$. Hence, a convolution of $\frac{1}{\omega_k} \sin(\omega_k t) e^{-r_k t}$ with $\tilde g_k(t)$ gives an approximation to the driven solution. For $\tilde g_k(t) = \cos(t)$ this approximation to the driven solution with homogeneous initial conditions is shown in Figure \ref{fig:driven_cwk_mode}, along with the exact solution (given by the last term of \ref{eq:time_domain_solution_of_cwk_for_modek}), with the convolutions computed using Theorem \ref{thm:Main_Convolution_theorem}.
\begin{figure}[t]
    \begin{center}
    \includegraphics[scale=.5]{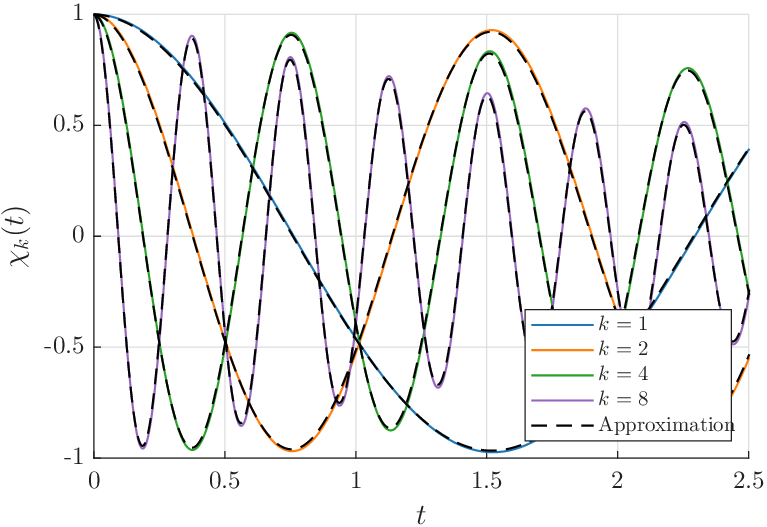}
    \caption{ 
    The exact and approximate time-domain representations of 
    $\chi_k(s) = \frac{s}{s^2 +\eta k^2 s^\alpha+ c^2k^2}$ for various values of $k$, with $\alpha=1/4$, $\eta=0.2$, and $c=2$. 
    Exact representations in time are given by the first term in \ref{eq:time_domain_solution_of_cwk_for_modek}, which is obtained through the convolution methods for Mittag-Leffler functions developed in this paper. The approximations are given by $e^{-r_kt} \cos(\omega_k t) - \frac{r_k}{\omega_k} e^{-r_k t}\sin(\omega_k t)$, where $r_k$ and $\omega_k$ are given in \ref{eq:rk_wk_definition}. The absolute difference for each $k$ case is of order $\eta^2$.}
    \label{fig:cwk_modes}
            \end{center}
\end{figure}

\begin{figure}[t]
    \begin{center}
    \includegraphics[scale=.5]{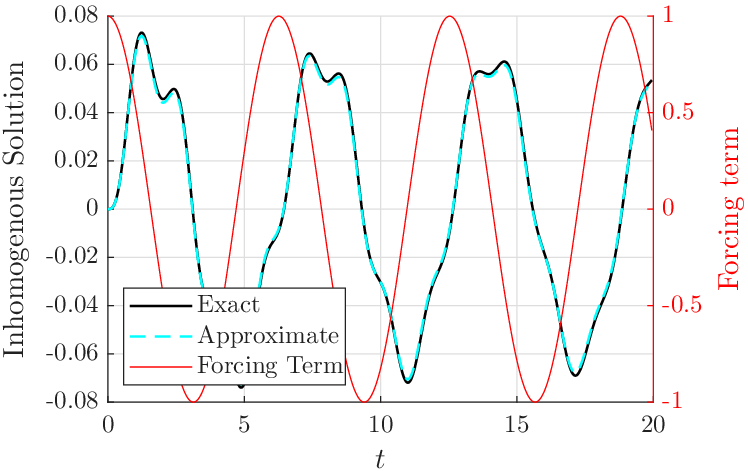}
    \caption{Exact and approximate solutions to $D^2_t f  + \eta k^2 D^\alpha_t  + c^2 k^2 = \cos( t)$, with $\alpha = 1/4$, $\eta=0.2$, $c=2$, $k=2$, and homogeneous initial conditions. The exact  and approximate solutions are given by $\El^{-1}\left[\frac{1}{s^2 + \eta k^2 s^\alpha + c^2 k^2}\right]$ and $e^{-r_k t} \sin (\omega_k t) $, respectively, convolved with $\cos(t).$  The relative error is less than 0.03.
    } \label{fig:driven_cwk_mode} \end{center}
\end{figure}
 


\section{Conclusion} \label{sec:5}

\setcounter{section}{5} \setcounter{equation}{0} 

In this paper we have demonstrated that convolutions of Mittag-Leffler type functions can be represented as sums of them. 
These functions characterize the eigenfunctions of Caputo and Riemann-Liouville fractional derivatives, and generalize the properties of exponentials and sinusoids.

We have shown that a convolution of two Mittag-Leffler type functions of the same order can be represented as a sum of two functions of the same type. This generalizes a well-known result for convolutions of exponential or trigonometric functions: a convolution of two exponentials or of two sinusoids is given as a sum of two exponentials or sinusoids, respectively.

We also have shown that when two Mittag-Leffler type functions have orders that differ by a  rational factor $n/m$, their convolution can be represented as a sum of $m+n$ Mittag-Leffler type functions. 
Additionally, in the case where the orders are not rationally related the convolution of the two functions can be represented as an infinite series of Mittag-Leffler type functions. 
Furthermore, we have shown that Mittag-Leffler type functions satisfy a generalized form of Euler's identity, relating functions of order $\alpha$ to those of order $n\alpha$.

These results are fundamental to the analysis of Mittag-Leffler functions and fractional derivatives. In addition, they provide an efficient tool for analyzing and solving inhomogeneous fractional differential equations, where convolutions of Mittag-Leffler type functions are common. We have demonstrated the strength of the obtained results by using them to represent the solutions to a forced anomalous diffusion equation, as well as the  Caputo-Wismer-Kelvin (fractional Kelvin-Voigt model) differential equation.

\appendix
\setcounter{equation}{0}
\renewcommand{\theequation}{\thesection.\arabic{equation}}
\makeatletter
\makeatletter
\renewcommand{\thetheorem}{\thesection.\arabic{theorem}}
\renewcommand{\thedefinition}{\thesection.\arabic{definition}}
\renewcommand{\thelemma}{\thesection.\arabic{lemma}}
\renewcommand{\thecorollary}{\thesection.\arabic{corollary}}
\makeatother

\makeatother
\setcounter{theorem}{0}

\section{ Appendix: Laplace transform} \label{secA1}
For convenience, we collect here information about properties of the Laplace transform; they are well known and can be found in many textbooks. 
\begin{definition}
Assuming that the integral in \ref{def:LaplaceTransform} converges, 
the Laplace transform of $f(t)$ is defined as
\label{def:LaplaceTransform}
    \begin{equation}
    \El[ f(t)](s) := \int_0^\infty f(t) e^{-st}dt.\label{eq:LaplaceTransform}
\end{equation}

\end{definition}
 A primary use of the Laplace transform is to transform  differential equations into algebraic equations, replacing derivatives in $t$ with factors of $s$.
 \begin{theorem}\label{thm:LaplaceOfD}
     Let $f$ have a Laplace transform, be continuous on $[0,\infty)$, and $f'$ be continuous on $(0,\infty)$. Then the Laplace transform of $f'(t)$ is given by
  \begin{equation}
    \El[D_tf](s) = s \El[f](s)-f(0) \label{eq:LaplaceOfD}.
\end{equation} \end{theorem}

\begin{proof}
    This is verified through integration by parts:
    \begin{align*}
    \int_0^\infty (D_t f) e^{-st} dt 
    &= f(t) e^{-st}\bigg|^{\infty}_0 - \int_0^\infty f (t)D_te^{-st} dt 
    \\=& 0 - f(0) -  (-s)\int_0^\infty f(t) e^{-st} dt
    \\&= s \El[f](s)-f(0)
    \end{align*} 
\end{proof}
Repeating the integration $n$ times gives the Laplace transform of the $n$th derivative of the function $f$:
\begin{corollary} \label{thm:LaplaceOfDn}
Let $f$ have a Laplace transform,  $\frac{d^kf}{dt^k}$ be continuous on $[0,\infty)$ for $k=0,\ldots,n-1$,  and $\frac{d^n}{dt^n}f$ be continuous on $(0,\infty)$. Then the Laplace transform of the $n$th derivative of $f$ is given by
\begin{equation}
    \El\left[D^nf\right](s) = s^n\El[f](s) - \sum_{\ell = 0}^{n-1} s^{n-1-\ell}f^{(\ell)}(0). \label{eq:LaplaceOfDn}
\end{equation}\end{corollary}

\begin{theorem}\label{thm:Derivative_in_Laplace_space}
    Let $f$ have a Laplace transform and be continuous on $(0,\infty)$. Then the Laplace transform of $t f(t)$ is given by
    \begin{equation}
        \El[t f(t)] = -\frac{d}{ds} \El[f](s) \label{eq:Derivative_in_Laplace_space}.
    \end{equation}
\end{theorem}
\begin{proof}
    Using the definition of the Laplace transform, we have \begin{align*}
        -\frac{d}{ds} \El[f](s) =&  -\frac{d}{ds}\int_0^\infty e^{-st} f(t) dt
        \\
        =&  \int_0^\infty e^{-st}t f(t) dt.
    \end{align*}
\end{proof}

The convolution operation used throughout this paper is the 'causal' convolution (as opposed to the convolution associated with Fourier analysis, which integrates from $-\infty$ to $\infty$).
\begin{definition}
    \label{def:CausalConvolution}
\begin{equation}
    (f*g)(t) :=\int_0^t f(t-\tau) g(\tau) d\tau \label{eq:CausalConvolution}
\end{equation}
\end{definition}

 This operation is commutative, which can be seen by a change of coordinates $\tau \rightarrow t-\tau$.
The fundamental property of the causal convolutions is that they are associated with products in the Laplace domain.
 \begin{theorem}
     Let $f(t)$ and $g(t)$ have Laplace transforms. Then the Laplace transform of $f*g$ is given by
\begin{equation}
    \El[f*g](s) = \El[f](s) \El[g](s) .\label{eq:ConvolutionTheorem}
\end{equation}

\label{thm:ConvolutionTheorem}
 \end{theorem}
\begin{proof}
    This follows from a change of coordinates:
\begin{align*}
    \El[f](s)\El[g](s) =& \int_0^\infty  f(t) e^{-st} dt \int_0^\infty   g(t') e^{-st'} dt'
    \\
    =& \int_0^\infty \int_0^\infty e^{-s(t+t')}f(t) g( t' ) dt dt'.
\end{align*}
Take $\xi = t+t'$. 
Changing the integration variables together with the limits of integration  
leads to 
\begin{align*}
    \El[f](s)\El[g](s) =& \int_0^\infty e^{-s\xi} \left(\int_0^\xi f(\xi-t')g(t')dt'\right)d\xi
    \\=&\El[f*g](s).
\end{align*}

\end{proof}
Finally, the paper relies on the Laplace transform of fractional powers of $t$:
\begin{lemma}
\label{lemma:LaplaceOfPower}
    For $a>-1$, the Laplace transform of $t^a$ is given by
    \begin{equation}
        \El[t^a](s) =s^{-1-a}
\Gamma(a+1).  \label{eq:LaplaceOfPower}  \end{equation}
\end{lemma}
\begin{proof}
    By definition \ref{def:LaplaceTransform},
    \begin{align*}
        \El[t^a] = \int_0^\infty t^a e^{-st}dt.
    \end{align*}
    Change coordinates from $t$ to $u= st$. Thus $dt = s^{-1} du$, and $t^a = u^as^{-a}$, and so the Laplace tranform of $t^a$ is
    \begin{equation}
       \El[t^a] =  s^{-1-a}\int_0^\infty  u^a e^{-u}du .
    \end{equation}
    Because $a>-1$ the integral converges and equals $\Gamma(a+1)$, by definition of the $
    \Gamma$ function. 
    In the case that $a\leq-1$, the integral fails to converge.
\end{proof}


\section*{\small Acknowledgments}
{\small  We gratefully acknowledge support from the Division of Mathematical 
Sciences at the US National Science Foundation (NSF) through Grants 
DMS-2111117  
and DMS-2136198.  
 }

 \section*{\small
 Conflict of interest} 

 {\small
 The authors declare no conflict of interest.}

\bigskip  

\small 

\end{document}